\documentclass[10pt]{amsart}
\usepackage{amssymb,amsmath}
\usepackage{epsfig}
\usepackage{hyperref}
\usepackage{appendix}
\newtheorem{theorem}{Theorem}[section]
\newtheorem{remm}{Remark}[section]
\newtheorem{lemma}{Lemma}[section]

\newtheorem{remark}[remm]{Remark}

\newtheorem{proposition}[theorem]{Proposition}

\numberwithin{equation}{section}

\newcommand{\VV}{{\sf V}}

\newcommand{\uu}{{\widehat u}}
\newcommand{\rset}{{\mathbb R}}
\newcommand{\nset}{{\mathbb N}}

\newcommand{\bfdot}{\bf\dot}

\newcommand{\ken}{\ \ }

\newcommand{\ssy}{\scriptscriptstyle}

\newcommand{\half}{\frac{1}{2}}
\newcommand{\dtau}{\Delta\tau}
\newcommand{\dt}{\Delta{t}}
\newcommand{\dx}{\Delta{x}}
\newcommand{\fem}{{\sf Z}_h^r}
\newcommand{\Nstar}{N_{\star}}
\newcommand{\Jstar}{J_{\star}}
\begin{document}
\title[]
{Crank-Nicolson finite element approximations\\
for a linear stochastic heat equation \\
with additive space-time white noise
}
%
%
%
%
%
%
\author[]
{Georgios E. Zouraris$^{\dag}$}
\thanks{%
$^{\dag}$Department of Mathematics and Applied Mathematics,
University of Crete, PO Box 2208, GR--710 03 Heraklion, Crete, Greece.
(e-mail: georgios.zouraris@uoc.gr)}
%
%
%
\maketitle
%
%
%
%
%
\begin{abstract}
We formulate an initial- and Dirichlet boundary- value problem for
a linear stochastic heat equation, in one space dimension, forced
by an additive space-time white noise.
First, we approximate the mild solution to the problem by the solution
of the regularized second-order linear stochastic parabolic
problem with random forcing proposed by Allen, Novosel and Zhang
(Stochastics Stochastics Rep., 64, 1998).
Then, we construct numerical approximations of the solution to the
regularized problem by combining the Crank-Nicolson method in
time with a standard Galerkin finite element method in space.
We derive strong a priori estimates of the modeling error 
made in approximating the mild solution to the problem
by the solution to the regularized problem, and of the numerical approximation
error of the Crank-Nicolson finite element method.
\end{abstract}
%
%
%
\section{Introduction}\label{SECT1}
%
%
\subsection{Formulation of the problem}
%
Let $T>0$, $D=(0,1)$ and $(\Omega,{\mathcal F},P)$ be a complete
probability space. Then, we consider the following initial- and
Dirichlet boundary- value problem for a linear stochastic heat
equation: find a stochastic function $u:[0,T]\times{\overline
D}\to\rset$ such that
\begin{equation}\label{PARAP}
\begin{gathered}
u_t=u_{xx}+{\dot W}\quad\text{\rm in}\quad(0,T]\times D,\\
u(t,\cdot)\big|_{\ssy\partial D}=0 \quad\forall\,t\in(0,T],\\
u(0,x)=0\quad\forall\,x\in D,\\
\end{gathered}
\end{equation}
a.s. in $\Omega$, where ${\dot W}$ denotes a space-time white
noise on $[0,T]\times D$ (see, e.g., \cite{Walsh86}, \cite{KXiong}).
The mild solution of the problem above is given by the formula
\begin{equation}\label{MildSol}
u(t,x)=\int_0^t\int_{\ssy D}G_{t-s}(x,y)\,dW(s,y),
\end{equation}
where $G_t(x,y)$ is the space-time Green kernel of the
solution to the deterministic parabolic problem: find a
deterministic function $v:[0,T]\times{\overline D}\to\rset$ such
that
\begin{equation}\label{Det_Parab}
\begin{gathered}
v_t =v_{xx}\quad\text{\rm in}\quad(0,T]\times D,\\
v(t,\cdot)\big|_{\ssy\partial D}=0 \quad\forall\,t\in(0,T],
\\
v(0,x)=v_0(x)\quad\forall\,x\in D,\\
\end{gathered}
\end{equation}
where $v_0:{\overline D}\rightarrow{\mathbb R}$ is a deterministic
initial condition. In particular, it holds that
\begin{equation}\label{deter_green}
v(t,x)=\int_{\ssy D}G_t(x,y)\,v_0(y)\,dy
\quad \forall\,(t,x)\in(0,T]\times{\overline D}
\end{equation}
and
\begin{equation}\label{GreenKernel}
G_t(x,y)=\sum_{k=1}^{\ssy\infty}e^{-\lambda_k^2 t}
\,\varepsilon_k(x)\,\varepsilon_k(y) \quad
\forall\,(t,x)\in(0,T]\times{\overline D},
\end{equation}
where $\lambda_k:=k\,\pi$ for $k\in\nset$, and
$\varepsilon_k(z):=\sqrt{2}\,\sin(\lambda_k\,z)$ for
$z\in{\overline D}$ and $k\in\nset$.
\par
The aim of the paper at hand is to analyze the convergence of
a Crank-Nicolson finite element method  for the approximation
of the mild solution $u$ to the problem \eqref{PARAP},
adapting properly the approach in \cite{Zouraris2017}.
%
\subsection{An approximate problem}
%
For $\Nstar$, $\Jstar\in\nset$, we define the
mesh-lengths $\dt:=\frac{T}{\Nstar}$ and
$\dx:=\frac{1}{J_{\star}}$, the nodes $t_n:=n\,\dt$
for $n=0,\dots,\Nstar$ and $x_j:=j\,\dx$ for $j=0,\dots,J_{\star}$,
and the subintervals $T_n:=(t_{n-1},t_n)$ and $D_j:=(x_{j-1},x_j)$
for $j=1,\dots,\Jstar$ and $n=1,\dots,\Nstar$. Also, for a nonempty set
$A$, we will denote by ${\mathcal X}_{\ssy A}$ the indicator function of $A$.
\par
Then, we consider the approximate stochastic parabolic problem
proposed in \cite{ANZ}, which is formulated as follows:
find stochastic function
${\widehat u}:[0,T]\times{\overline D}\rightarrow{\mathbb R}$ such that
\begin{equation}\label{AC2}
\begin{gathered}
{\widehat u}_t={\widehat u}_{xx}+{\mathcal W}
\quad\text{\rm in}\ken (0,T]\times D,\\
{\widehat u}(t,\cdot)\big|_{\ssy\partial D}=0
\quad\forall\,t\in(0,T],\\
{\widehat u}(0,x)=0\quad\forall\,x\in D,\\
\end{gathered}
\end{equation}
a.e. in $\Omega$, where ${\mathcal W}\in L^2((0,T)\times D)$
is a piecewise constant stochastic function given by
\begin{equation}\label{Dnoise_Def}
{\mathcal W}:=
\tfrac{1}{\dt\,\dx}\,
\sum_{n=1}^{\ssy\Nstar}\sum_{j=1}^{\ssy\Jstar}
{\mathcal R}_j^n\,{\mathcal X}_{\ssy T_n\times D_j},
\end{equation}
where
\begin{equation}\label{R_Def}
{\mathcal R}_j^n:=W(T_n\times D_j),
\quad\,n=1,\dots,\Nstar,\quad j=1,\dots,\Jstar.
\end{equation}
The standard theory of parabolic problems (see, e.g., \cite{LMag}),
yields
\begin{equation}\label{HatUform}
\widehat{u}(t,x)= \int_0^t\int_{\ssy D} G_{t-s}(x,y)
\,{\mathcal W}(s,y)\,dsdy
\quad\forall\,(t,x)\in[0,T]\times{\overline D},
\end{equation}
from which we conclude that:
the stochastic function ${\widehat u}$ has a finite noise structure depending on $\Nstar\Jstar$
random variables (see Remark~\ref{Rimarkos1}), and the space-time regularity of ${\widehat u}$
is higher than that of the mild solution $u$ to \eqref{PARAP}.
%
%
%
\begin{remark}\label{Rimarkos1}
If $S\subset [0,T]\times{\overline D}$ then $W(S)\sim N(0,|S|)$.
Thus, for $j=1,\dots,\Jstar$ and $n=1,\dots,\Nstar$, we have
$W(T_n\times D_j)\sim N(0,\dt\dx)$. Also,
if $S_1,S_2\subset [0,T]\times{\overline D}$ and $S_1\cap S_2=\varnothing$,
then $W(S_1)$ and $W(S_2)$ are independent random variables. Thus, the
random variables $(({\mathcal R}_j^n)_{j=1}^{\ssy\Jstar})_{n=1}^{\ssy\Nstar}$ are
independent. For further details we refer
the reader to \cite{Walsh86} and \cite{KXiong}.
\end{remark}
%
%
%
\subsection{The numerical method}\label{kefalaiaki}
\par
Let $M\in\nset$, $\dtau:=\tfrac{T}{M}$,
$\tau_m:=m\,\dtau$ for $m=0,\dots,M$ and
$\Delta_m:=(\tau_{m-1},\tau_m)$ for $m=1,\dots,M$.
Also, let $\fem\subset H_0^1(D)$ be a finite element space consisting of
functions which are piecewise polynomials of degree at most $r$
over a partition of $D$ in intervals with maximum mesh-length $h$.
Then, we constuct computable fully-discrete approximations of
${\widehat u}$ using the Crank-Nicolson method for time-stepping
and the standard Galerkin finite element method for space discretization.
The proposed method is as follows: first we set
\begin{equation}\label{FullDE1}
U_h^0:=0,
\end{equation}
and then, for $m=1,\dots,M$, we are seeking $U_h^m\in S_h^r$ such
that
\begin{equation}\label{FullDE2}
\left(U_h^m-U_h^{m-1},\chi\right)_{\ssy 0,D}
+\tfrac{\dtau}{2}\,\left(\partial(U_h^{m}+U_h^{m-1}),\partial\chi\right)_{\ssy 0,D}
=\int_{\ssy\Delta_m} ({\mathcal W},\chi)_{\ssy 0,D}\,ds
\quad\forall\,\chi\in S_h^r,
\end{equation}
where $(\cdot,\cdot)_{\ssy 0,D}$ is the usual $L^2(D)$ inner product.
%
\subsection{References and main results of the paper}
%
%
%
%
The Crank-Nicolson method for stochastic evolution equations has been analyzed
in \cite{Erika02} and \cite{Erika03} under the assumption that the additive space-time
noise is smooth in space.
In \cite{Walsh05},  the solution of a semilinear stochastic heat equation with
multiplicative noise is approximated by a numerical method  that combines
a $\theta-$time discretization method with a finite element method in space
over a uniform mesh.
Also, for the proposed family of methods an error analysis has been developed,
the outcome  of which depends on the value of $\theta$. In particular,
for the Backward Euler finite element method $(\theta=1)$, 
the error estimate obtained is of the form $O\big(\dtau^{\frac{1}{4}}+h^{\frac{1}{2}}\big)$,
which is in agreement with that in \cite{YubinY05} for the stochastic heat equation with
multiplicative noise.
However, for the Crank-Nicolson finite element method $(\theta=\tfrac{1}{2})$,
the corresponding error estimate is of the form $O\big(\dtau\,h^{-\frac{3}{2}}+h^{\half}\big)$,
which means that convergence follows when $\dtau=o(h^{\frac{3}{2}})$.
Taking into account that both methods must have the same order of convergence since the regularity
of the solution is low, one is wondering if the difference between the above error estimates
is caused by the fact that the Crank-Nicolson method is not strongly stable
while the Backward Euler method is.
Here, adapting properly the analysis in \cite{Zouraris2017}, we show that,
in the special case of an additive space time white noise, is that obtained
in \cite{YubinY05} for the Backward Euler finite element method and thus
no mesh conditions are required to ensure its convergence.
\par
The first step in our analysis is to show that  the solution ${\widehat u}$
to the regularized problem \eqref{AC2} converges to the mild solution $u$
to the problem \eqref{PARAP} when both $\dx$ and $\dt$ tend freely to zero.
This is established by proving in Theorem~\ref{model_theorem} the following
$L^{\infty}(L^2_{\ssy P}(L^2_x))$ {\it modeling error} bound:
\begin{equation*}
\max_{t\in[0,T]}\left(\,{\mathbb E}\left[
\|u(t,\cdot)-{\widehat u}(t,\cdot)\|_{\ssy L^2(D)}^2\right]\right)^{\ssy 1/2}
\leq\,C\,\left(\,\epsilon^{-\frac{1}{2}}\,\dx^{\frac{1}{2}-\epsilon}
+\dt^{\frac{1}{4}}\,\right)
\end{equation*}
for $\epsilon\in\left(0,\tfrac{1}{2}\right]$, which measures the effect of breaking the infinite
noise dimension of $u$ to the finite one of ${\widehat u}$.
%
%
\par
The second step is the estimation of the discretization error of
the Crank-Nicolson finite element method for the approximation of ${\widehat u}$
that is  formulated in Section~\ref{kefalaiaki}. We achieve the derivation of a discrete in time
$L^{\infty}_t(L^2_{\ssy P}(L^2_x))$ error estimate of the following
form:
\begin{equation*}
\max_{0\leq{m}\leq{\ssy M}}\left(\,
{\mathbb E}\left[\|U_h^m-{\widehat u}(\tau_m,\cdot)\|_{\ssy L^2(D)}\right]\right)^{\ssy 1/2}
\leq\,{\sf C}\,\left(\,
\epsilon_1^{-\frac{1}{2}}\,\dtau^{\frac{1}{4}-\epsilon_1}
+\epsilon_2^{-\frac{1}{2}}\,h^{\frac{1}{2}-\epsilon_2}\,\right)
\end{equation*}
for $\epsilon_1\in\left(0,\tfrac{1}{4}\right]$ and $\epsilon_2\in\left(0,\tfrac{1}{2}\right]$,
where the constant ${\sf C}$ is independent of $\dt$ and $\dx$.
To get the latter estimate, we use a discrete Duhamel principle (cf. \cite{YubinY05}, \cite{BinLi})
that is based on the representation of the numerical approximations via the outcome
of a modified Crank-Nicolson method for the deterministic problem. The analysis is moving along the lines of the
convergence analysis developed in \cite{Zouraris2017} for a Crank-Nicolson finite
element method proposed to approximate the mild solution of a fourth order, linear
parabolic problem with additive space time white noise.
%
%
\par
Let us give an epigrammatic description of the paper.
In Section~\ref{Section2} we introduce notation and recall
several known results.
In Section~\ref{Section3} we estimate the modeling error.
In Section~\ref{Section4} we introduce and analyze the convergence
of modified Crank-Nicolson time-discrete and fully-discrete approximations
of $v$.
Finally, Section~\ref{Section5} contains the error analysis of the
Crank-Nicolson finite element approximations of ${\widehat u}$
formulated in Section~\ref{kefalaiaki}.
%
%
%
%
%
%
%
\section{Preliminaries}\label{Section2}
We denote by $L^2(D)$ the space of the Lebesgue measurable
functions that are square integrable on $D$ with respect to
the Lebesgue measure $dx$, equipped with the standard norm
$\|g\|_{\ssy 0,D}:= (\int_{\ssy D}|g(x)|^2\,dx)^{\ssy 1/2}$
for $g\in L^2(D)$, which is derived by the inner product
$(g_1,g_2)_{\ssy 0,D}
:=\int_{\ssy D}g_1(x)g_2(x)\,dx$ for $g_1$, $g_2\in L^2(D)$.
\par
For $s\in\nset_0$, we denote by $H^s(D)$ the Sobolev space of functions
having generalized derivatives up to order $s$ in $L^2(D)$,
by $\|\cdot\|_{\ssy s,D}$ its usual norm, i.e. $\|g\|_{\ssy s,D}:=(\sum_{\ell=0}^s
\|\partial^{\ell}g\|_{\ssy 0,D}^2)^{\ssy 1/2}$ for $g\in H^s(D)$,
and by $|\cdot|_{\ssy s,D}$ the seminorm $|g|_{\ssy s,D}:=
\|\partial^sg\|_{\ssy 0,D}$ for $g\in H^s(D)$.
Also, by $H_0^1(D)$ we denote the subspace of $H^1(D)$
consisting of functions that vanish at the endpoints of $D$ in
the sense of trace, where the, well-known,
Poincar{\'e}-Friedrich inequality holds, i.e.,
\begin{equation}\label{Poincare}
\|g\|_{\ssy 0,D}\leq\,C_{\ssy P\!F}\,|g|_{\ssy 1,D}
\quad\forall\,g\in H^1_0(D).
\end{equation}
\par
The sequence of pairs
$\left\{\big(\lambda_i^2,\varepsilon_i\big)\right\}_{i=1}^{\infty}$ is a
solution to the eigenvalue/eigenfunction problem: find nonzero
$z\in H^2(D)\cap H_0^1(D)$ and $\lambda\in\rset$ such that
$-z''=\lambda\,z$ in $D$. Since
$(\varepsilon_i)_{i=1}^{\infty}$ forms a complete orthonormal system
in $\left(\,L^2(D),(\cdot,\cdot)_{\ssy 0,D}\,\right)$, for
$s\in\rset$, we define the following subspace of $L^2(D)$
\begin{equation*}
{\mathcal V}^s(D):=\left\{g\in L^2(D):\quad\sum_{i=1}^{\ssy\infty}
\lambda_{i}^{2s} \,(g,\varepsilon_i)^2_{\ssy 0,D}<\infty\,\right\}
\end{equation*}
provided with the natural norm
$\|g\|_{\ssy{\mathcal V}^s}:=\big(\,\sum_{i=1}^{\infty}
\lambda_{i}^{2s}\,(g,\varepsilon_i)^2_{\ssy 0,D}
\,\big)^{\ssy 1/2}$ for $g\in{\mathcal V}^s(D)$.
For $s\ge 0$, the space $({\mathcal V}^s(D),\|\cdot\|_{\ssy{\mathcal V}^s})$ is a complete
subspace of $L^2(D)$ and we define
$({\bfdot H}^s(D),\|\cdot\|_{\ssy{\bfdot H}^s})
:=({\mathcal V}^s(D),\|\cdot\|_{\ssy{\mathcal V}^s})$.
For $s<0$, the space $({\bfdot H}^s(D),\|\cdot\|_{\ssy{\bfdot H}^s})$  is defined
as the completion of $({\mathcal V}^s(D),\|\cdot\|_{\ssy{\mathcal V}^s})$, or,
equivalently, as the dual of $({\bfdot H}^{-s}(D),\|\cdot\|_{\ssy{\bfdot H}^{-s}})$.
\par
Let $m\in\nset_0$. It is well-known (see, e.g., \cite{Thomee}) that
\begin{equation}\label{dot_charact}
{\bfdot H}^m(D)=\left\{\,g\in H^m(D):
\quad\partial^{2i}g\left|_{\ssy\partial D}\right.=0
\quad\text{\rm if}\quad 0\leq{2\,i}<m\,\right\}
\end{equation}
and there exist constants $C_{m,{\ssy A}}$ and $C_{m,{\ssy B}}$
such that
\begin{equation}\label{H_equiv}
C_{m,{\ssy A}}\,\|g\|_{\ssy m,D} \leq\|g\|_{\ssy{\bfdot H}^m}
\leq\,C_{m,{\ssy B}}\,\|g\|_{\ssy m,D}
\quad \forall\,g\in{\bfdot H}^m(D).
\end{equation}
Also, we define on $L^2(D)$ the negative norm
$\|\cdot\|_{\ssy -m, D}$ by
\begin{equation*}
\|g\|_{\ssy -m, D}:=\sup\left\{ \tfrac{(g,\varphi)_{\ssy 0,D}}
{\|\varphi\|_{\ssy m,D}}:\quad \varphi\in{\bfdot H}^m(D)
\quad\text{\rm and}\quad\varphi\not=0\right\} \quad\forall\,g\in
L^2(D),
\end{equation*}
for which, using \eqref{H_equiv}, it is easy to conclude that
there exists a constant $C_{-m}>0$ such that
\begin{equation}\label{minus_equiv}
\|g\|_{\ssy -m,D}\leq\,C_{-m}\,\|g\|_{{\bfdot H}^{-m}}
\quad\forall\,g\in L^2(D).
\end{equation}
\par
Let ${\mathbb L}_2=(L^2(D),(\cdot,\cdot)_{\ssy 0,D})$ and
${\mathcal L}({\mathbb L}_2)$ be the space of linear, bounded
operators from ${\mathbb L}_2$ to ${\mathbb L}_2$. An
operator $\Gamma\in {\mathcal L}({\mathbb L}_2)$ is
Hilbert-Schmidt, when $\|\Gamma\|_{\ssy\rm
HS}:=\left(\,\sum_{i=1}^{\infty} \|\Gamma\varepsilon_i\|^2_{\ssy
0,D}\,\right)^{\ssy 1/2}$ is finite, where $\|\Gamma\|_{\ssy\rm HS}$ is
the so called Hilbert-Schmidt norm of $\Gamma$.
%
%
We note that the quantity $\|\Gamma\|_{\ssy\rm HS}$ does not
change when we replace $(\varepsilon_i)_{i=1}^{\infty}$ by
another complete orthonormal system of ${\mathbb L}_2$.
It is well known (see, e.g., \cite{DunSch}, \cite{Lord}) that an operator
$\Gamma\in{\mathcal L}({\mathbb L}_2)$ is Hilbert-Schmidt iff
there exists a measurable function $\gamma:D\times D\rightarrow{\mathbb
R}$ such that $\Gamma[v](\cdot)=\int_{\ssy D}\gamma(\cdot,y)\,v(y)\,dy$
for $v\in L^2(D)$, and then, it holds that
\begin{equation}\label{HSxar}
\|\Gamma\|_{\ssy\rm HS} =\left(\int_{\ssy D}\int_{\ssy
D}\gamma^2(x,y)\,dxdy\right)^{\ssy 1/2}.
\end{equation}
Let ${\mathcal L}_{\ssy\rm HS}({\mathbb L}_2)$ be the set of
Hilbert Schmidt operators of ${\mathcal L}({\mathbb L}^2)$ and
$\Phi\in L^2([0,T];{\mathcal L}_{\ssy\rm HS}({\mathbb L}_2))$.
Also, for a random variable $X$, let ${\mathbb E}[X]$ be its
expected value, i.e., ${\mathbb E}[X]:=\int_{\ssy\Omega}X\,dP$.
Then, the It{\^o} isometry property for stochastic integrals reads
\begin{equation}\label{Ito_Isom}
{\mathbb E}\left[\left\|\int_0^{\ssy T}\Phi\,dW\right\|_{\ssy 0,D}^2\right]
=\int_0^{\ssy T}\|\Phi(t)\|_{\ssy\rm HS}^2\,dt.
\end{equation}
\par
We recall that (see Appendix A in \cite{KZ2009}): if $c_{\star}>0$, then
\begin{equation}\label{SR_BOUND}
\sum_{i=1}^{\ssy\infty}\tfrac{1}{\lambda_i^{1+c_{\star}\nu}}
\leq\,\left(\tfrac{1+2c_{\star}}{c_{\star}\pi}\right)\,\nu^{-1}
\quad\forall\,\nu\in(0,2],
\end{equation}
and if $({\mathcal H},(\cdot,\cdot)_{\ssy{\mathcal H}})$ is a
real inner product space, then
\begin{equation}\label{innerproduct}
(g-v,g)_{\ssy{\mathcal H}}
\ge\,\tfrac{1}{2}\,\left[\,(g,g)_{\ssy{\mathcal H}}
-(v,v)_{\ssy{\mathcal H}}\,\right]\quad\forall\,g,v\in{\mathcal H}.
\end{equation}
\par
Also, for any $L\in{\mathbb N}$ and
functions $(v^{\ell})_{\ell=0}^{\ssy L}\subset L^2(D)$ we set
$v^{\ell-\half}:=\tfrac{1}{2}(v^{\ell}+v^{\ell-1})$
for $\ell=1,\dots,L$. Finally, for $\alpha\in[0,1]$ and for $n=0,\dots,M-1$,
we set $\tau_{n+\alpha}:=\tau_n+\alpha\,\dtau$.
%
\subsection{A projection operator}
%
Let ${\mathcal O}:=(0,T)\times D$,
${\mathfrak G}_{\ssy\Nstar}:=\text{\rm span}({\mathcal X}_{\ssy T_n})_{n=1}^{\ssy\Nstar}$,
${\mathfrak G}_{\ssy\Jstar}:=\text{\rm span}({\mathcal X}_{\ssy D_j})_{j=1}^{\ssy\Jstar}$
and ${\sf\Pi}:L^2({\mathcal O})\rightarrow{\mathfrak G}_{\ssy\Nstar}\otimes{\mathfrak G}_{\ssy\Jstar}$
be defined by
\begin{equation}\label{Defin_L2}
{\sf\Pi}{g}(s,y):=\tfrac{1}{\dt\,\dx}\,\sum_{n=1}^{\ssy\Nstar}
\sum_{j=1}^{\ssy\Jstar}\left(\int_{\ssy T_n}\int_{\ssy D_j}
g(t,x)\,{dx}dt\right)\,{\mathcal X}_{\ssy T_n}(s)\,{\mathcal X}_{\ssy D_j}(y)
\quad\forall\,(s,y)\in{\mathcal O},
\end{equation}
which is the usual $L^2({\mathcal O})-$projection onto the finite dimensional space
${\mathfrak G}_{\ssy\Nstar}\otimes{\mathfrak G}_{\ssy\Jstar}$.
In the lemma below, we connect the projection operator ${\sf\Pi}$ to
the stochastic load ${\mathcal W}$ (cf. Lemma~2.1 in \cite{KZ2010},
\cite{ANZ}).
%
%
%
%
\begin{lemma}\label{Lhmma1}
For $g\in L^2({\mathcal O})$, it holds that
\begin{equation}\label{WNEQ2}
\int_0^{\ssy T}\int_{\ssy D}\,{\sf\Pi}{g}(s,y)\,dW(s,y)
=\int_0^{\ssy T}\int_{\ssy D}{\mathcal W}(t,x)\,g(t,x)\,dxdt.
\end{equation}
\end{lemma}
%
%
%
%
%
%
\begin{proof}
Using the properties of the stochastic intergral (see, e.g., \cite{Walsh86}) 
along with \eqref{Defin_L2}, \eqref{Dnoise_Def} and \eqref{R_Def}, we obtain
\begin{equation*}
\begin{split}
\int_0^{\ssy T}\int_{\ssy D}{\sf\Pi}g(s,y)\,dW(s,y)
%
=&\tfrac{1}{\dt\,\dx}\,\sum_{n=1}^{\ssy\Nstar}\sum_{j=1}^{\ssy\Jstar}
\left(\int_{\ssy T_n}\int_{\ssy D_j}g(t,x)\;dxdt\right)\,W(T_n\times D_j)\\
%
=&\int_0^{\ssy T}\int_{\ssy D}{\mathcal W}(t,x)\,g(t,x)\,dxdt.
\end{split}
\end{equation*}
\end{proof}
%
%
%
%
Since it holds that
\begin{equation*}
\int_0^{\ssy T}\int_{\ssy D}{\sf\Pi}g(s,y)\,\varphi\,dyds=\int_0^{\ssy T}\int_{\ssy D}g(s,y)\,\varphi(s,y)\,dsdy
\quad\forall\,\varphi\in{\mathfrak G}_{\ssy\Nstar}\otimes{\mathfrak G}_{\ssy\Jstar},
\quad\forall\,g\in L^2({\mathcal O}),
\end{equation*}
we, easily, conclude that
\begin{equation}\label{peykos1}
\int_0^{\ssy T}\int_{\ssy D}({\sf\Pi}g(s,y))^2\,dsdy\leq\,\int_0^{\ssy T}\int_{\ssy D}(g(s,y))^2\,dsdy
\quad\forall\,g\in L^2({\mathcal O}).
\end{equation}
%
%
\subsection{Linear elliptic and parabolic operators}
For given $f\in L^2(D)$, let $T_{\ssy E}f\in {\bfdot H}^2(D)$ be the solution
to the problem: find $v_{\ssy E}\in{\bfdot H}^2(D)$ such that
\begin{equation}\label{TwoPoint}
v_{\ssy E}''=f\quad\text{\rm in}\quad D,
\end{equation}
i.e. $T_{\ssy E}f:=v_{\ssy E}$. It is well-known that
\begin{equation}\label{TB-prop1}
(v_1,T_{\ssy E}v_2)_{\ssy 0,D}
=-\left(\,\partial(T_{\ssy E}v_1),\partial(T_{\ssy E}v_2\,)\right)_{\ssy 0,D}
=(T_{\ssy E}v_1,v_2)_{\ssy 0,D} \quad\forall\,v_1,v_2\in L^2(D),
\end{equation}
and, for $m\in{\mathbb N}_0$, there exists a constant $C_{\ssy E}^m>0$ such that
\begin{equation}\label{ElReg1}
\|T_{\ssy E}f\|_{\ssy m,D}\leq \,C_{\ssy E}^m\,\|f\|_{\ssy m-2, D},
\quad\forall\,f\in H^{\max\{0,m-2\}}(D).
\end{equation}
%
%
\par
Let $\left({\mathcal S}(t)v_0\right)_{t\in[0,T]}$ be the standard semigroup
notation for the solution $v$ of \eqref{Det_Parab}. Then
(see, e.g., \cite{Thomee}) for $\ell\in{\mathbb N}_0$, $\beta\ge0$, $p\ge0$ and
$q\in\left[0,p+2\ell\right]$, there exists a constant $C_{p,q,\ell}>0$ such
that
\begin{equation}\label{Reggo0}
\big\|\partial_t^{\ell}{\mathcal S}(t)v_0\big\|_{\ssy {\bfdot H}^p}
\leq \,C_{p,q,\ell}\,\,\,t^{-\frac{p-q}{2}-\ell}
\,\,\,\|v_0\|_{\ssy {\bfdot
H}^q}\quad\forall\,t>0,\ken\forall\,v_0\in{\bfdot H}^q(D),
\end{equation}
and a constant $C_{\beta}>0$ such that
\begin{equation}\label{Reggo3}
\int_{t_a}^{t_b}(s-t_a)^{\beta}\,
\big\|\partial_t^{\ell}{\mathcal S}(s)v_0 \big\|_{\ssy {\bfdot H}^p}^2\,ds
\leq\,C_{\beta}\, \|v_0\|^2_{\ssy {\bfdot H}^{p+2\ell-\beta-1}}
\end{equation}
for all $v_0\in{\bfdot H}^{p+2\ell-\beta-1}(D)$ and $t_a$, $t_b\in[0,T]$ with $t_a<t_b$.
\subsection{Discrete operators}
Let $r\in{\mathbb N}$ and $\fem\subset H_0^1(D)$ be a finite
element space consisting of functions which are
piecewise polynomials of degree at most $r$ over a partition of $D$ in
intervals with maximum mesh-length $h$. It is well-known (cf.,
e.g., \cite{Cia}, \cite{BrScott}) that there exists a constant
$C_{\ssy\rm FE}>0$ such that
\begin{equation}\label{pin0}
\inf_{\chi\in S_h^r}\|g-\chi\|_{\ssy 1,D}
\leq\,C_{\ssy\rm FE}\,h\,\|g\|_{\ssy 2,D}
\quad\,\forall\,g\in H^2(D)\cap H_0^1(D).
\end{equation}
Then, we define the discrete Laplacian operator
$\Delta_h:\fem\to\fem$ by
$(\Delta_h\varphi,\chi)_{\ssy 0,D} =(\partial\varphi,\partial\chi)_{\ssy 0,D}$,
for $\varphi,\chi\in\fem$,
the $L^2(D)-$projection operator $P_h:L^2(D)\to\fem$ by
$(P_hf,\chi)_{\ssy 0,D}=(f,\chi)_{\ssy 0,D}$
for $\chi\in\fem$ and $f\in L^2(D)$,
and the standard Galerkin finite element approximation
$v_{{\ssy E},h}\in\fem$ of the solution $v_{\ssy E}$ to \eqref{TwoPoint} is
specified by requiring
\begin{equation}\label{fem1}
-\Delta_hv_{{\ssy E},h}=P_hf.
\end{equation}
\par
Let $T_{{\ssy E},h}:L^2(D)\to\fem$ be the solution operator
of the finite element method \eqref{fem1}, i.e. $T_{{\ssy E},h}f:=v_{{\ssy
E},h}=-\Delta_h^{-1}P_hf$ for $f\in L^2(D)$. Then, we can easily conclude
that
\begin{equation}\label{Frako1}
(f,T_{{\ssy E},h}g)_{\ssy 0,D}=
-(\partial(T_{{\ssy E},h}f),\partial(T_{{\ssy E},h}g))_{\ssy 0,D}
=(g,T_{{\ssy E},h}f)_{\ssy 0,D}\quad\forall\,f,g\in L^2(D),
\end{equation}
which, along with \eqref{Poincare}, yields
\begin{equation}\label{DELSTAB}
\|T_{{\ssy E},h}f\|_{\ssy 1,D}\leq\,C\, \|f\|_{\ssy -1, D}
\quad\forall\,f\in L^2(D).
\end{equation}
Due to the approximation property \eqref{pin0},
the theory of the standard Galerkin finite
element method for second order elliptic problems (cf., e.g.,
\cite{Cia}, \cite{BrScott}), yields that
\begin{equation}\label{pin1}
\begin{split}
\|T_{\ssy E}f-T_{{\ssy E},h}f\|_{\ssy 0,D}
\leq&\,C\,h^2\,\|T_{\ssy E}f\|_{\ssy 2,D}\\
\leq&\,C\,h^2\,\|f\|_{\ssy 0,D},\quad\forall\,f\in
L^2(D).\\
\end{split}
\end{equation}
%
%
\section{Estimating the Modeling Error}\label{Section3}
Here, we derive an  $L^{\infty}_t(L^2_{\ssy P}(L^2_x))-$estimate of the modeling error
in terms of $\dt$ and $\dx$ (cf.  \cite{ANZ}, \cite{BinLi}, \cite{KZ2008}, \cite{KZ2010}).
%
%
\begin{theorem}\label{model_theorem}
Let $u$ be the mild solution to \eqref{PARAP},
${\widehat u}$ be the solution to \eqref{AC2}
and ${\mathcal Z}(t):=\left({\mathbb E}\,\left[
\|u(t,\cdot)-{\widehat u}(t,\cdot)\|_{\ssy 0,D}^2\right]
\right)^{\ssy 1/2}$ for $t\in[0,T]$.
Then, there exist a constant
$C_{\ssy\rm ME}>0$, independent of $\dt$ and $\dx$, such that
\begin{equation}\label{ModelError}
\max_{\ssy t\in[0,T]}{\mathcal Z}(t)\leq\,C_{\ssy\rm ME}\,\left(\,\dt^\frac{1}{4}
+\epsilon^{-\frac{1}{2}}\,\dx^{\frac{1}{2}-\epsilon}\,\right)
\quad\forall\,\epsilon\in\left(0,\tfrac{1}{2}\right].
\end{equation}
\end{theorem}
%
%
%
%
%
%
\begin{proof} The proof is moving along the lines of the proof of
Theorem~3.1 in \cite{KZ2010}. To simplify the notation we set
$S_{n,j}:=T_n\times D_j$ for $n=1,\dots,\Nstar$ and $j=1,\dots,\Jstar$.
Also, we will use the 
symbol $C$ to denote a generic constant that is independent of
$\dt$ and $\dx$, and may changes value from the one line to the other.
\par
Let $t\in(0,T]$. Using \eqref{MildSol}, \eqref{HatUform}, Lemma~\ref{Lhmma1} 
and It{\^o} isometry, we conclude that
%
%
\begin{equation*}
{\mathcal Z}(t)=\left(\int_0^{\ssy T}\left(\,
\int_{\ssy D}\int_{\ssy D}|\Psi_{t,x}(s,y)|^2\,dxdy\,\right)\,ds\right)^{\ssy 1/2},
\end{equation*}
where $\Psi_{t,x}(s,y):=g_{t,x}(s,y)-{\sf\Pi}g_{t,x}(s,y)$ and
$g_{t,x}(s,y):={\mathcal X}_{\ssy(0,t)}(s)\,G_{t-s}(x,y)$.
Now, we introduce the following splitting
\begin{equation}\label{arokaria0}
{\mathcal Z}(t) \leq\,{\mathcal Z}_{1}(t)+{\mathcal Z}_{2}(t)
\end{equation}
where
${\mathcal Z}_{\ell}(t):=\left(\,\int_0^{\ssy T}\left(
\int_{\ssy D}\int_{\ssy D}|\Psi_{t,x}^{\ell}(s,y)|^2\,dxdy
\right)\,ds\right)^{1/2}$ for $\ell=1,2$, and
\begin{equation*}
\begin{split}
\Psi_{t,x}^1(s,y):=&\,\tfrac{1}{\dt\,\dx}
\int_{\ssy S_{n,j}}{\mathcal X}_{(0,t)}(s)\,
\left[G_{t-s}(x,y)-G_{t-s}(x,y')\right]\,dy'ds',\\
\Psi_{t,x}^2(s,y):=&\,\tfrac{1}{\dt\,\dx}
\int_{\ssy S_{n,j}}
\left[{\mathcal X}_{(0,t)}(s)\,G_{t-s}(x,y')
-{\mathcal X}_{(0,t)}(s')G_{t-s'}(x,y')\right]\,dy'ds'\\
\end{split}
\end{equation*}
for $(s,y)\in S_{n,j}$ and for $n=1,\dots,\Nstar$ and $j=1,\dots,\Jstar$.
%
%
%
\par\noindent
\vskip0.3truecm
\par
{\tt Estimation of ${\mathcal Z}_{\ssy 1}(t)$}:  Using \eqref{GreenKernel}
and the $L^2(D)-$ orthogonality of $(\varepsilon_k)_{k=1}^{\ssy\infty}$
we have
\begin{equation*}
\begin{split}
{\mathcal Z}^2_{1}(t)=&\,\tfrac{1}{\dx^2}
\sum_{n=1}^{\ssy\Nstar}\sum_{j=1}^{\ssy\Jstar}
\int_{\ssy D}\left[\int_{\ssy S_{n,j}}
\left(\int_{\ssy D_j} {\mathcal X}_{(0,t)}(s)
\,\left[G_{t-s}(x,y)-G_{t-s}(x,y')\right]\,dy'\right)^2\,dyds
\right]\,dx\\
=&\,\tfrac{1}{\dx^2} \sum_{n=1}^{\ssy\Nstar}\sum_{j=1}^{\ssy\Jstar}
\int_{\ssy S_{n,j}}{\mathcal X}_{(0,t)}(s)\left(
\sum_{k=1}^{\infty}e^{-2\lambda_k^2(t-s)}
\,\left(\int_{\ssy D_j}(\varepsilon_k(y)-\varepsilon_k(y'))\,dy'\right)^2\right)\,dyds\\
=&\,\tfrac{1}{\dx^2}\sum_{k=1}^{\infty}
\left(\int_0^t
e^{-2\lambda_k^2(t-s)}\,ds\right)\,
\left(\sum_{j=1}^{\ssy\Jstar}\int_{\ssy D_j}\left(
\int_{\ssy D_j}(\varepsilon_k(y)-\varepsilon_k(y'))\,dy'\right)^2\,dy\right),\\
\end{split}
\end{equation*}
which, along with the Cauchy-Schwarz inequality, yields
\begin{equation}\label{arokaria1}
{\mathcal Z}^2_1(t)\leq \sum_{k=1}^{\infty}
\left(\int_0^te^{-2\lambda_k^2(t-s)}\,ds\right)
\,
\left(\tfrac{1}{\dx}\sum_{j=1}^{\ssy\Jstar}
\int_{\ssy D_j\times D_j}\left(\varepsilon_k(y)-\varepsilon_k(y')\right)^2\,dy'dy\right).
\end{equation}
\par
Let $k\in{\mathbb N}$ and $j\in\{\,1,\dots,\Jstar\,\}$. Using
the mean value theorem we have
\begin{equation}\label{arokaria2}
\begin{split}
\sup_{y,y'\in D_j}\big|\varepsilon_k(y)-\varepsilon_k(y')\big|^2
&\leq\,2\,\min\left\{1,\lambda_k^2\,\dx^2\right\}\\
&\leq\,2\,\min\left\{1,\lambda_k^2\,\dx^2\right\}^{\gamma}\\
&\leq\,2\,\lambda_k^{2\gamma}\,\dx^{2\gamma}
\quad\forall\,\gamma\in[0,1], \\
\end{split}
\end{equation}
and
\begin{equation}\label{arokaria3}
\begin{split}
\int_0^te^{-2\lambda_k^2(t-s)}\,ds &=\,\tfrac{1}{2\,\lambda_k^2}\,
\big(1-e^{-2\lambda_k^2t}\big)\\
&\leq\,\tfrac{1}{2\,\lambda_k^2}\cdot\\
\end{split}
\end{equation}
%
%
%
%
\par
Let $\gamma\in\left[0,\frac{1}{2}\right)$ and
$\epsilon=\frac{1}{2}-\gamma\in\left(0,\frac{1}{2}\right]$.
We combine \eqref{arokaria0}, \eqref{arokaria1},
\eqref{arokaria2} and \eqref{SR_BOUND} to obtain
\begin{equation}\label{arokaria4}
\begin{split}
{\mathcal Z}_1(t)\leq&\,\tfrac{\dx^{\gamma}}{\pi^{(1-\gamma)}}
\left(\,\sum_{k=1}^{\infty}\tfrac{1}{k^{1+2(\frac{1}{2}-\gamma)}}\,\right)^{\ssy 1/2}\\
\leq&\,C\,\dx^{\frac{1}{2}-\epsilon}\,\epsilon^{-\frac{1}{2}}.\\
\end{split}
\end{equation}
\par\noindent
\vskip0.3truecm
\par
{\tt Estimation of ${\mathcal Z}_2(t)$}:
Using again \eqref{GreenKernel} and the $L^2(D)-$ orthogonality of $(\varepsilon_k)_{k=1}^{\ssy\infty}$
we have
%
%
\begin{equation*}
{\mathcal Z}_2(t)=\left(\,\sum_{k=1}^{\ssy\infty} \Upsilon_1^k\,\Upsilon_2^k\,\right)^{\ssy 1/2},
\end{equation*}
where
\begin{equation*}
\begin{split}
\Upsilon^k_1:=&\,\tfrac{1}{\dx}\,\sum_{j=1}^{\ssy\Jstar}\left(
\int_{\ssy D_j}\varepsilon_k(y')\,dy'\right)^2,\\
\Upsilon^k_2:=&\,\tfrac{1}{\dt^2}\,\sum_{n=1}^{\ssy\Nstar}\int_{\ssy T_n}
\left(\,\int_{\ssy T_n}\left[{\mathcal X}_{(0,t)}(s)\,e^{-\lambda_k^2(t-s)} -{\mathcal
X}_{(0,t)}(s')\,e^{-\lambda_k^2(t-s')}\right]\,ds'\,\right)^2\,ds.\\
\end{split}
\end{equation*}
Observing that
\begin{equation*}
\begin{split}
\Upsilon_1^k\leq&\,\sum_{j=1}^{\ssy\Jstar}\int_{\ssy D_j}\varepsilon_k^2(y')\;dy'\\
\leq&\,\int_{\ssy D}\varepsilon_k^2(y')\;dy'\\
\leq&\,1,\quad\forall\,k\in{\mathbb N},
\end{split}
\end{equation*}
and letting $N(t)\in\left\{1,\dots,\Nstar\right\}$ such that $t\in\left(t_{\ssy N(t)-1},t_{\ssy N(t)}\right]$
we conclude that
\begin{equation}\label{arodamos1}
{\mathcal Z}_2(t)\leq\left(\,
\sum_{k=1}^{\ssy\infty}\sum_{n=1}^{\ssy N(t)}\Psi_n^k
\,\right)^{\ssy 1/2},
\end{equation}
where
\begin{equation*}
\Psi_n^k:=\tfrac{1}{\dt^2} \int_{\ssy T_n}
\left(\,\int_{\ssy T_n}\left({\mathcal X}_{(0,t)}(s)
\,e^{-\lambda_k^2(t-s)} -{\mathcal
X}_{(0,t)}(s')\,e^{-\lambda_k^2(t-s')}\right)\,ds'\,\right)^2\,ds.
\end{equation*}
\par
Let $k\in{\mathbb N}$ and $n\in\{1,\dots,N(t)-1\}$. Then, we have
\begin{equation*}
\begin{split}
\Psi_n^k=&\,\tfrac{1}{\dt^2}\, \int_{\ssy T_n}
\left(\int_{\ssy T_n}\left(\int_s^{s'} \lambda_k^2\,
e^{-\lambda_k^2(t-\tau)}\,d\tau\right)\,ds'\right)^2\,ds\\
\leq&\,\tfrac{1}{\dt^2}\, \int_{\ssy T_n}\left(\int_{\ssy T_n}
\left(\int_{\ssy t_{n-1}}^{\max\{s',s\}} \lambda_k^2\,
e^{-\lambda_k^2(t-\tau)}\,d\tau\right)\,ds'\right)^2\,ds\\
\leq&\,\tfrac{2}{\dt^2}\, \int_{\ssy T_n}\left(
\int_{\ssy T_n}\left(\int_{t_{n-1}}^{s'} \lambda_k^2\,
e^{-\lambda_k^2(t-\tau)}\,d\tau\right)\,ds'\right)^2\,ds\\
&\,\quad+\tfrac{2}{\dt^2}\, \int_{\ssy T_n}\left(
\int_{\ssy T_n}\left(\int_{t_{n-1}}^{s}\lambda_k^2\,
e^{-\lambda_k^2(t-\tau)}\,d\tau\right)\,ds'\right)^2\,ds,\\
%
\end{split}
\end{equation*}
from which, using the Cauchy-Schwarz inequility, we obtain
\begin{equation*}
\begin{split}
\Psi_n^k\leq&\,4\,\int_{\ssy T_n}\left(\int_{\ssy t_{n-1}}^{s}
\lambda_k^2\,e^{-\lambda_k^2\,(t-\tau)}\,d\tau\right)^2\,ds\\
&\leq\,4\,\int_{\ssy T_n}
\left(e^{-\lambda_k^2(t-s)}-e^{-\lambda_k^2\,(t-t_{n-1})}\right)^2\,ds\\
&\leq\,4\, \int_{\ssy T_n}e^{-2\lambda_k^2\,(t-s)}
\left(1-e^{-\lambda_k^2\,(s-t_{n-1})}\right)^2\,ds\\
&\leq\,4\,\left(1-e^{-\lambda_k^2\,\dt}\right)^2
\int_{\ssy T_n}e^{-2\lambda_k^2\,(t-s)}\,ds\\
&\leq\,2\,\left(1-e^{-\lambda_k^2\,\dt}\right)^2
\,\tfrac{e^{-\lambda_k^2\,(t-t_n)}
-e^{-\lambda_k^2\,(t-t_{n-1})}}{\lambda_k^2}\cdot\\
\end{split}
\end{equation*}
Thus, by summing with respect to $n$, we obtain
\begin{equation}\label{arodamos2}
\sum_{n=1}^{\ssy N(t)-1}\Psi_n^k\leq\,2\,
\tfrac{\big(1-e^{-\lambda_k^2\,\dt}\big)^2}{\lambda_k^2}\cdot
\end{equation}
Also, we have
\begin{equation*}
\begin{split}
%
\Psi_{\ssy N(t)}^k&= \tfrac{1}{\dt^2} \int_{t_{N(t)-1}}^t
\left(\int_{t_{N(t)-1}}^t\left(
\int_{s'}^s\lambda_k^2\,e^{-\lambda_k^2\,(t-\tau)}\,d\tau\right)\,ds'
+\int_t^{t_{N(t)}} e^{-\lambda_k^2\,(t-s)}\,ds'
\right)^2\,ds\\
&\hskip2.0truecm
+\tfrac{1}{\dt^2}\int_t^{t_{N(t)}}\left(
\int_{t_{N(t)}-1}^t e^{-\lambda_k^2\,(t-s')}\,ds'\right)^2\,ds\\
%
%
&\leq \tfrac{1}{\dt^2} \int_{t_{N(t)-1}}^t
\left(\int_{t_{N(t)-1}}^t
\left(\int_{s'}^s\lambda_k^2\,e^{-\lambda_k^2\,(t-\tau)}\,d\tau\right)\,ds'
+\dt\,e^{-\lambda_k^2\,(t-s)}\right)^2\,ds\\
&\hskip2.0truecm +\tfrac{1}{\dt}\,
\tfrac{\big(1-e^{-\lambda_k^2\,(t-t_{N(t)-1})}\big)^2}{\lambda_k^4}\\
%
%
&\leq \tfrac{2}{\dt^2} \int_{t_{N(t)-1}}^t \left(
\int_{t_{N(t)-1}}^t
\left(\int_{s'}^s\lambda_k^2\,e^{-\lambda_k^2\,(t-\tau)}\,d\tau\right)\,ds'
\right)^2\,ds \\
&\hskip1.0truecm
+\tfrac{\big(1-e^{-2\lambda_k^2\,(t-t_{N(t)-1})}\big)^2}{\lambda_k^2}
+\tfrac{1}{\dt}\,
\tfrac{\big(1-e^{-\lambda_k^2\,(t-t_{N(t)-1})}\big)^2}{\lambda_k^4}\\
%
%
&\leq\tfrac{2}{\dt^2} \int_{t_{N(t)-1}}^t
\left(\int_{t_{N(t)-1}}^t
\left(\int_{t_{N(t)-1}}^{\max\{s,s'\}}\lambda_k^2
\,e^{-\lambda_k^2\,(t-\tau)}\,d\tau\right)\,ds'
\right)^2\,ds\\
&\hskip1.0truecm
+\tfrac{\big(1-e^{-2\lambda_k^2\,\dt}\big)^2}{\lambda_k^2}
+\tfrac{1}{\dt}\,
\tfrac{\big(1-e^{-\lambda_k^2\,\dt}\big)^2}{\lambda_k^4}\\
%
%
&\leq \,8\, \int_{t_{N(t)-1}}^t
\left(\int_{t_{N(t)-1}}^s\lambda_k^2\,e^{-\lambda_k^2\,(t-\tau)}\,d\tau\right)^2\,ds
+\tfrac{\big(1-e^{-2\lambda_k^2\,\dt}\big)^2}{\lambda_k^2}
+\tfrac{1}{\dt}\,
\tfrac{\big(1-e^{-\lambda_k^2\,\dt}\big)^2}{\lambda_k^4}\\
%
&\leq \,8\, \int_{t_{N(t)-1}}^t
\left(e^{-\lambda_k^2(t-s)}-e^{-\lambda_k^2(t-t_{N(t)-1})}\right)^2\,ds
+\tfrac{\big(1-e^{-2\lambda_k^2\,\dt}\big)^2}{\lambda_k^2}
+\tfrac{1}{\dt}\,
\tfrac{\big(1-e^{-\lambda_k^2\,\dt}\big)^2}{\lambda_k^4},\\
\end{split}
\end{equation*}
which finally gives
\begin{equation}\label{arodamos3}
\Psi_{\ssy N(t)}^k\leq
\,5\,\tfrac{\big(1-e^{-2\lambda_k^2\,\dt}\big)^2}{\lambda_k^2}
+\tfrac{1}{\dt}\,
\tfrac{\big(1-e^{-\lambda_k^2\,\dt}\big)^2}{\lambda_k^4}\cdot
\end{equation}
\par
Combining \eqref{arodamos1}, \eqref{arodamos2} and \eqref{arodamos3} we obtain
\begin{equation}\label{arodamos4}
{\mathcal Z}_2(t)\leq\,\left(
\,7\,\sum_{k=1}^{\ssy\infty}\tfrac{\big(1-e^{-2\lambda_k^2\,\dt}\big)^2}{\lambda_k^2}
+\tfrac{1}{\dt}\,\sum_{k=1}^{\ssy\infty}
\tfrac{\big(1-e^{-\lambda_k^2\,\dt}\big)^2}{\lambda_k^4}\,\right)^{\ssy 1/2}\cdot
\end{equation}
The last step in the proof is to bound the series above in terms of $\dt$.
For the first series, we proceed as follows
\begin{equation*}
\begin{split}
\sum_{k=1}^{\ssy\infty}
\tfrac{\big(\,1-e^{-2\,\lambda_k^2\,\dt}\,\big)^2}{\lambda_k^2}
\leq&\, \tfrac{\big(\,1-e^{-2\,\pi^2\,\dt}\,\big)^2}{\pi^2}
+\int_1^{\ssy\infty} \tfrac{\big(\,1-e^{-2\,x^2\,\pi^2\,\dt}\,\big)^2}
{x^2\,\pi^2}\,dx\\
\leq&\,\tfrac{2}{\pi^2}\,
\big(\,1-e^{-2\,\pi^2\,\dt}\,\big)^2
+8\,\dt\,\int_1^{\infty}
(\,1-e^{-2\,x^2\,\pi^2\,\dt}\,)
\,e^{-2\,x^2\,\pi^2\,\dt}\,dx\\
\leq&\,\tfrac{2}{\pi^2}\,\left(\int_0^1
\big(-e^{-2\,x\,\pi^2\,\dt}\big)' \,dx\right)^2
+\tfrac{8}{\pi}\,\sqrt{\dt}\,
\int_{\pi\,\ssy\sqrt{\dt}}^{\ssy\infty} e^{-2y^2}\,dy\\
\leq&\,\tfrac{2}{\pi^2}\, \left(\,2\,\pi^2\,\dt\,
\int_0^1 e^{-2\,x\,\pi^2\,\dt}\,dx\,\right)^2
+\tfrac{8}{\pi}\,\sqrt{\dt}\,
\left(\,\int_0^1e^{-2y^2}\,dy+\int_1^{\ssy\infty}e^{-2y^2}\,dy\,\right)\\
\leq&\,8\,\pi^2\,\dt^2
+\tfrac{8}{\pi}\,\sqrt{\dt}\,
\left(1+\tfrac{1}{2}\,\int_1^{\ssy\infty}\tfrac{1}{y^2}\,dy\right)\\
\end{split}
\end{equation*}
from which we obtain
\begin{equation}\label{arodamos5}
\sum_{k=1}^{\ssy\infty} \tfrac{\big(\,1-e^{-2\lambda_k^2\,\dt}\,\big)^2}
{\lambda_k^2}\leq\,C\,\sqrt{\dt}.
\end{equation}
We treat the second series, in a similar manner, as follows
\begin{equation*}
\begin{split}
\sum_{k=1}^{\ssy\infty} \tfrac{\big(\,1-e^{-\lambda_k^2\,\dt}\,\big)^2}
{\lambda_k^4}
\leq&\,\tfrac{\big(\,1-e^{-\pi^2\,\dt}\,\big)^2}{\pi^4}
+\int_1^{\ssy\infty} \tfrac{\big(\,1-e^{-x^2\,\pi^2\,\dt}\,\big)^2}
{x^4\,\pi^4}\,dx\\
\leq&\,\tfrac{4}{3\,\pi^4}\,
\big(\,1-e^{-\pi^2\,\dt}\,\big)^2
+\tfrac{4\,\dt}{3\,\pi^2}\,\int_1^{\ssy\infty}
\tfrac{\big(\,1-e^{-x^2\,\pi^2\,\dt}\,\big)
\,e^{-x^2\,\pi^2\,\dt}}{x^2}\,dx\\
\leq&\,\tfrac{4}{3\,\pi^4}\,\left(\int_0^1
\big(e^{-x\,\pi^2\,\dt}\big)' \,dx\right)^2
+\tfrac{4}{3\,\pi^2}\,\dt\,\int_1^{\ssy\infty}
\tfrac{1-e^{-x^2\,\pi^2\,\dt}}{x^2}\,dx\\
\leq&\,\tfrac{4}{3\,\pi^4}\,\left(\pi^2\,\dt\,
\int_0^1e^{-x\,\pi^2\,\dt}\,dx\right)^2\\
&\quad+\tfrac{4}{3\,\pi^2}\,\dt\,\left(1-e^{-\pi^2\,\dt}
+2\,\pi^2\,\dt\,\int_1^{\ssy\infty}e^{-x^2\,\pi^2\,\dt}\,dx\right)\\
\leq&\,\tfrac{8}{3}\,\dt^2
+\tfrac{8}{3}\,\dt^{\frac{3}{2}}
\,\int_{\pi\,\ssy\sqrt{\dt}}^{\ssy\infty}e^{-y^2}\,dy\\
\leq&\,\tfrac{8}{3}\,\dt^2
+\tfrac{8}{3}\,\dt^{\frac{3}{2}}
\,\left(\int_0^1e^{-y^2}\,dy+\int_1^{\ssy\infty}e^{-y^2}\,dy\right)\\
\leq&\,\tfrac{8}{3}\,\dt^2
+\tfrac{8}{3}\,\dt^{\frac{3}{2}}
\,\left(1+\int_1^{\infty}\tfrac{1}{y^2}\,dy\right),\\
\end{split}
\end{equation*}
from which we arrive at
\begin{equation}\label{arodamos6}
\tfrac{1}{\dt}\,\sum_{k=1}^{\ssy\infty}
\tfrac{\big(\,1-e^{-\lambda_k^2\,\dt}\,\big)^2}{\lambda_k^4}
\leq\,C\,\sqrt{\dt}.
\end{equation}
Using the bounds \eqref{arodamos4}, \eqref{arodamos5} and \eqref{arodamos6}
we arrive at
\begin{equation}\label{arodamos7}
{\mathcal Z}_2(t)\leq\,C\,\dt^{\frac{1}{4}}.
\end{equation}
\par
Since ${\mathcal Z}(0)=0$, the error bound \eqref{ModelError} follows easily from
\eqref{arokaria0}, \eqref{arokaria4} and \eqref{arodamos7}.
\end{proof}
\begin{remark}
In \cite{ANZ} and \cite{BinLi} is obtained an $O(\dt^{\frac{1}{4}}+\dx\,\dt^{-\frac{1}{4}})$
a priori estimate of the modelling error measured in the $L^2_{\ssy P}(L^2_t(L^2_x))$ norm, which
introduces the need to assume a CFL condition in order to conclude a rate of convergence,
when $\dt,\dx\rightarrow 0$.
The $L^{\ssy\infty}_t(L^2_{\ssy P}(L^2_x))-$modelling error estimate derived in
Theorem~\ref{model_theorem} is valid, without requiring any mesh condition
between $\dt$ and $\dx$.
\end{remark}
%
%
%
%
%
\section{A Modified Crank-Nicolson Method for the Deterministic Problem}\label{Section4}
%
%
Following \cite{Zouraris2017}, we introduce and analyze modified Crank-Nicolson
time-discrete and fully-discrete approximations of the solution to the deterministic
problem \eqref{Det_Parab}, which are necessary to carry out the convergence
analysis of the Crank-Nicolson finite element method defined in
Section~\ref{kefalaiaki}.
%
%
\subsection{Time-Discrete approximations}\label{Section41}
%
%
The modified Crank-Nicolson time-discrete approximations
of the solution $v$ to \eqref{Det_Parab} follow, first, by setting
\begin{equation}\label{CNDet1}
\VV^0:=v_0
\end{equation}
and by finding ${\VV}^1\in{\bfdot H}^2(D)$ such that
\begin{equation}\label{CNDet11}
\VV^1-\VV^0=\tfrac{\dtau}{2}\,\partial^2\VV^1,
\end{equation}
and then, for $m=2,\dots,M$, by specifying $\VV^m\in{\bfdot H}^2(D)$
satisfying
\begin{equation}\label{CNDet2}
\VV^m-\VV^{m-1}=\dtau\,\,\partial^2\VV^{m-\frac{1}{2}}.
\end{equation}
\par
The first convergence result we provide, is a discrete in time  $L^2_t(L^2_x)$ estimate
of time-averages of the nodal error.
%
%
%
\begin{proposition}\label{DetPropo1}
Let $(\VV^m)_{m=0}^{\ssy M}$ be the time-discrete approximations of the solution
$v$ to the problem \eqref{Det_Parab} defined by \eqref{CNDet1}--\eqref{CNDet2}.
Then, there exists a constant $C>0$, independent of $\dtau$, such that
\begin{equation}\label{Ydaspis900}
\left(\,\dtau\sum_{m=1}^{\ssy M}
\big\|\VV^{m-\frac{1}{2}}-v^{m-\frac{1}{2}}\big\|_{\ssy
0,D}^2 \,\right)^{\ssy 1/2} \leq\,C\,\dtau^{\theta}
\,\|v_0\|_{\ssy{\bfdot H}^{2\theta-1}}
\quad\forall\,\theta\in[0,1],\quad\forall\,v_0\in{\bfdot H}^1(D),
\end{equation}
where $v^{\ell}(\cdot):=v(\tau_{\ell},\cdot)$ for $\ell=0,\dots,M$.
\end{proposition}
%
%
%
%
%
%
%
%
\begin{proof}
In the sequel, we will use the symbol $C$ to denote a generic
constant that is independent of $\dtau$
and may changes value from the one line to the other.
\par
Let  ${\mathcal E}^{\frac{1}{2}}:=v(\tau_{\frac{1}{2}},\cdot)-\VV^1$
and ${\mathcal E}^m:=v^m-\VV^m$ for
$m=0,\dots,M$.
Then, using \eqref{Det_Parab} and \eqref{CNDet2}, we conclude that
\begin{equation}\label{CityHall2010_3}
T_{\ssy E}({\mathcal E}^m-{\mathcal E}^{m-1})=\dtau\,{\mathcal E}^{m-\frac{1}{2}}
+\rho_m,\quad m=2,\dots,M,
\end{equation}
where
$\rho_m(\cdot):=\int_{\ssy\Delta_m}
[v(s,\cdot)-v^{m-\frac{1}{2}}(\cdot)]\,ds$.
Taking the $L^2(D)-$inner product of both sides of
\eqref{CityHall2010_3} with ${\mathcal E}^{m-\frac{1}{2}}$,
using \eqref{TB-prop1} and then summing
with respect to $m$ from $2$ up to $M$, we obtain
\begin{equation*}
|T_{\ssy E}{\mathcal E}^{\ssy M}|_{\ssy 1,D}^2
-|T_{\ssy E}{\mathcal E}^1|_{\ssy 1,D}^2
+2\,\dtau\,\sum_{m=2}^{\ssy M}\|{\mathcal E}^{m-\frac{1}{2}}\|_{\ssy 0,D}^2
=-2\,\sum_{m=2}^{\ssy M}(\rho_m,{\mathcal E}^{m-\frac{1}{2}})_{\ssy 0,D},
\end{equation*}
which, after applying the Cauchy-Schwarz inequality
and the geometric mean inequality, yields
%
%
\begin{equation}\label{May2015_3}
\dtau\,\sum_{m=2}^{\ssy M}\|{\mathcal E}^{m-\half}\|_{\ssy 0,D}^2
\leq\,|T_{\ssy E}{\mathcal E}^{1}|^2_{\ssy 1,D}
+\dtau^{-1}\,\sum_{m=2}^{\ssy M}\|\rho_m\|_{\ssy 0,D}^2.
\end{equation}
Next, we bound the residual functions $(\rho_m)_{m=2}^{\ssy M}$ as follows:
\begin{equation}\label{May2015_4}
\begin{split}
\|\rho_m\|_{\ssy 0,D}^2=&\,\tfrac{1}{4}\,\int_{\ssy D}
\left(-\int_{\ssy\Delta_m}\int_{\tau}^{\tau_m}
\partial_{\tau}v(s,x)\,ds{d\tau}
+\int_{\ssy\Delta_m}\int_{\tau_{m-1}}^{\tau}
\partial_{\tau}v(s,x)\,ds{d\tau}\right)^2\,dx\\
\leq&\,\int_{\ssy D}\left(\int_{\ssy\Delta_m}
\int_{\ssy\Delta_m}|\partial_{\tau}v(s,x)|\,ds
d\tau\right)^2\,dx\\
\leq&\,\dtau^3\,\int_{\ssy\Delta_m}
\|\partial_{\tau}v(s,\cdot)\|_{\ssy 0,D}^2\,ds,
\quad m=2,\dots,M.\\
\end{split}
\end{equation}
Also, observing that ${\mathcal E}^0=0$ and combining \eqref{May2015_3},
\eqref{May2015_4} and \eqref{Reggo3} (with $\beta=0$, $\ell=1$, $p=0$),
we obtain
\begin{equation}\label{May2015_5}
\begin{split}
\dtau\,\sum_{m=1}^{\ssy M}\|{\mathcal E}^{m-\frac{1}{2}}\|_{\ssy
0,D}^2\leq&\,\tfrac{\dtau}{4}\,\|{\mathcal E}^1\|_{\ssy 0,D}^2
+|T_{\ssy E}{\mathcal E}^{1}|^2_{\ssy 1,D}
+\dtau^2\,
\int_0^{\ssy T}\|\partial_{\tau}v(s,\cdot)\|_{\ssy 0,D}^2\,ds\\
\leq&\,\tfrac{\dtau}{4}\,\|{\mathcal E}^1\|_{\ssy 0,D}^2
+|T_{\ssy E}{\mathcal E}^{1}|^2_{\ssy 1,D}
+C\,\dtau^2\,\|v_0\|_{\ssy{\bfdot H}^1}^2.\\
\end{split}
\end{equation}
In order to bound the first two terms in the right hand side of
\eqref{May2015_5}, we introduce the following splittings
\begin{gather}
|T_{\ssy E}{\mathcal E}^1|_{\ssy 1,D}^2\leq\,2\,\left(\,
|T_{\ssy E}(v(\tau_1,\cdot)-v(\tau_{\half},\cdot))|^2_{\ssy 1,D}
+|T_{\ssy E}{\mathcal E}^{\frac{1}{2}}|^2_{\ssy 1,D}\,\right),
\label{May2015_6}\\
\dtau\,\|{\mathcal E}^1\|_{\ssy 0,D}^2\leq\,2\,\dtau\,\left(\,
\|v(\tau_1,\cdot)-v(\tau_{\half},\cdot)\|_{\ssy 0,D}^2
+\|{\mathcal E}^{\frac{1}{2}}\|_{\ssy 0,D}^2\,\right).\label{May2015_6a}
\end{gather}
We continue by estimating the terms in the right hand side of \eqref{May2015_6}
and \eqref{May2015_6a}. First, we observe that
%
%
$\|v(\tau_1,\cdot)-v(\tau_{\half},\cdot)\|_{\ssy 0,D}^2
\leq\tfrac{\dtau}{2}\,\int_{\tau_{\half}}^{\tau_1}
\|\partial_{\tau}v(\tau,\cdot)\|^2_{\ssy 0,D}\;d\tau$,
which, along with \eqref{Reggo3} (with $\ell=1$, $p=0$, $\beta=0$), yields
\begin{equation}\label{May2015_7}
\|v(\tau_1,\cdot)-v(\tau_{\half},\cdot)\|_{\ssy 0,D}^2
\leq\,C\,\dtau\,\|v_0\|^2_{\ssy{\bfdot H}^1}.
\end{equation}
Next, we use \eqref{ElReg1} and \eqref{minus_equiv}, to get
\begin{equation*}\label{May2015_8a}
\begin{split}
|T_{\ssy E}(v(\tau_1,\cdot)-v(\tau_{\half},\cdot))|_{\ssy 1,D}^2
\leq&\,\tfrac{\dtau}{2}\,\int_{\tau_{\half}}^{\tau_1}
|T_{\ssy E}\left(\partial_{\tau}v(\tau,\cdot)\right)|_{\ssy 1,D}^2\;d\tau\\
\leq&\,C\,\dtau\,\int_{\tau_{\half}}^{\tau_1}
\|\partial_{\tau}v(\tau,\cdot)\|_{\ssy -1,D}^2\;d\tau\\
\leq&\,C\,\dtau\,\int_{\tau_{\half}}^{\tau_1}
\|\partial_{\tau}v(\tau,\cdot)\|^2_{\ssy{\bfdot H}^{-1}}\;d\tau.\\
\end{split}
\end{equation*}
Observing that
$v(\tau,\cdot)=\sum_{i=1}^{\infty}e^{-\lambda_i^2\,\tau}
\,(v_0,\varepsilon_i)_{\ssy 0,D}\,\varepsilon_i(\cdot)$ for $\tau\in[0,T]$,
we obtain
\begin{equation*}\label{May2015_8b}
\begin{split}
\|\partial_{\tau}v(\tau,\cdot)\|_{\ssy{\bfdot H}^{-1}}^2
=&\,\sum_{i=1}^{\infty}\lambda_i^{-2}
\,\,|(\partial_{\tau}v(\tau,\cdot),\varepsilon_i)_{\ssy 0,D}|^2\\
=&\,\sum_{i=1}^{\infty}\lambda_i^2\,e^{-2\,\lambda_i^2\,\tau}
\,|(v_0,\varepsilon_i)_{\ssy 0,D}|^2\\
\leq&\,\|v_0\|_{\ssy{\bfdot H}^1}^2\quad\forall\,\tau\in[0,T].\\
\end{split}
\end{equation*}
Thus,  we arrive at
\begin{equation}\label{May2015_9}
|T_{\ssy E}(v(\tau_1,\cdot)-v(\tau_{\half},\cdot))|_{\ssy 1,D}^2
\leq\,C\,\dtau^2\,\|v_0\|_{\ssy{\bfdot H}^1}^2.
\end{equation}
Finally, using \eqref{Det_Parab} and \eqref{CNDet11} we have
\begin{equation}\label{May2015_10}
T_{\ssy E}({\mathcal E}^{\frac{1}{2}}-{\mathcal E}^0)
=\tfrac{\dtau}{2}\,{\mathcal E}^{\frac{1}{2}}+\rho_{\frac{1}{2}}
\end{equation}
with
$\rho_{\frac{1}{2}}(\cdot):=\int_{0}^{\tau_{\frac{1}{2}}}
[\,v(s,\cdot)-v(\tau_{\half},\cdot)]\;ds$.
Since ${\mathcal E}^0=0$, after taking the $L^2(D)-$inner product of both sides
of \eqref{May2015_10} with ${\mathcal E}^{\frac{1}{2}}$ and using \eqref{TB-prop1} and
the Cauchy-Schwarz inequality along with the arithmetic mean inequality, we obtain
\begin{equation}\label{May2015_12}
|T_{\ssy E}{\mathcal E}^{\frac{1}{2}}|_{\ssy 1,D}^2
+\tfrac{\dtau}{2}\,\|{\mathcal E}^{\frac{1}{2}}\|_{\ssy 0,D}^2
\leq\tfrac{1}{\dtau}\,\|\rho_{\frac{1}{2}}\|_{\ssy 0,D}^2
+\tfrac{\dtau}{4}\,\|{\mathcal E}^{\frac{1}{2}}\|_{\ssy 0,D}^2.
\end{equation}
Now, using  \eqref{Reggo3} (with $\beta=0$, $\ell=1$, $p=0$) we obtain
\begin{equation*}
\begin{split}
\|\rho_{\frac{1}{2}}\|_{\ssy 0,D}^2=&\,\int_{\ssy D}
\left[\,\int_0^{\tau_{\half}}\left(\int_s^{\tau_{\half}}\partial_{\tau}v(\tau,x)\;d\tau\right)\;ds
\,\right]^2\;dx\\
\leq&\,\tfrac{\dtau^3}{8}
\int_0^{\tau_1}\|\partial_{\tau}v(\tau,\cdot)\|^2_{\ssy 0,D}\;d\tau\\
\leq&\,C\,\dtau^3\,\|v_0\|_{\ssy{\bfdot H}^1}^2,\\
\end{split}
\end{equation*}
which, along with \eqref{May2015_12}, yields
\begin{equation}\label{May2015_13}
|T_{\ssy E}{\mathcal E}^{\frac{1}{2}}|_{\ssy 1,D}^2
+\tfrac{\dtau}{4}\,\|{\mathcal E}^{\frac{1}{2}}\|_{\ssy 0,D}^2
\leq\,C\,\dtau^2\,\|v_0\|^2_{\ssy{\bfdot H}^1}.
\end{equation}
Thus, from  \eqref{May2015_5}, \eqref{May2015_6}, \eqref{May2015_6a},
\eqref{May2015_7}, \eqref{May2015_9} and \eqref{May2015_13}, we conclude that
\eqref{Ydaspis900} holds for $\theta=1$.
\par
We continue, by observing that \eqref{CNDet11} and \eqref{CNDet2} are equivalent
to
\begin{gather}
T_{\ssy E}(\VV^1-\VV^{0})=\tfrac{\dtau}{2}\,\VV^{1},\label{May2015_20a}\\
T_{\ssy E}(\VV^m-\VV^{m-1})=\dtau\,\VV^{m-\frac{1}{2}},\quad
m=2,\dots,M.\label{May2015_20}
\end{gather}
Next, we take the $L^2(D)-$inner product of both sides of
\eqref{May2015_20} with $\VV^{m-\frac{1}{2}}$,
use \eqref{TB-prop1} and sum with respect to $m$ from
$2$ up to $M$, to obtain
%
%
%
\begin{equation}\label{May2015_21}
\dtau\,\|\VV^1\|_{\ssy 0,D}^2
+\dtau\,\sum_{m=2}^{\ssy M}\|\VV^{m-\frac{1}{2}}\|_{\ssy 0,D}^2\leq
\,\dtau\,\|\VV^1\|_{\ssy 0,D}^2+|T_{\ssy E}\VV^1|_{\ssy 1,D}^2.
\end{equation}
Now, we take the $L^2(D)-$inner product of both sides of
\eqref{May2015_20a} with $\VV^{1}$, use \eqref{TB-prop1} along
with \eqref{innerproduct} and \eqref{CNDet1} to get
\begin{equation}\label{May2015_22}
|T_{\ssy E}\VV^1|_{\ssy 1,D}^2+\dtau\,\|\VV^1\|_{\ssy 0,D}^2
\leq\,|T_{\ssy E}v^0|_{\ssy 1,D}^2.
\end{equation}
Combining \eqref{May2015_21} and \eqref{May2015_22} and then
using \eqref{ElReg1} and \eqref{minus_equiv}, we obtain
\begin{equation}\label{May2015_23}
\begin{split}
\dtau\,\|\VV^1\|_{\ssy 0,D}^2
+\dtau\,\sum_{m=2}^{\ssy M}\|\VV^{m-\frac{1}{2}}\|_{\ssy 0,D}^2
\leq&\,C\,\|v_0\|_{\ssy -1,D}^2\\
\leq&\,C\,\|v_0\|_{\ssy{{\bfdot H}^{-1}}}^2.\\
\end{split}
\end{equation}
In addition, we have
\begin{equation}\label{Museum_1}
\dtau\,\|v^1\|^2_{\ssy 0,D}+\dtau\sum_{m=2}^{\ssy M}\|v^{m-\half}\|_{\ssy 0,D}^2
\leq\,2\,\dtau\,\sum_{m=1}^{\ssy M}\|v^m\|_{\ssy 0,D}^2
\end{equation}
and
\begin{equation*}
\begin{split}
2\,\dtau\,\sum_{m=1}^{\ssy M}\|v^m\|_{\ssy 0,D}^2
\leq&\,2\,\dtau^{-1}\,\sum_{m=1}^{\ssy M}\int_{\ssy D}
\left(\,\int_{\tau_{m-1}}^{\tau_m}
\partial_{\tau}\left[\,(\tau-\tau_{m-1})
\,v(\tau,x)\,\right]\,d\tau\,\right)^2\,dx\\
\leq&\,2\,\dtau^{-1}\,\sum_{m=1}^{\ssy M}\int_{\ssy D}
\left(\,\int_{\tau_{m-1}}^{\tau_m}\left[\,v(\tau,x)
+(\tau-\tau_{m-1})\,v_{\tau}(\tau,x)\,\right]
\,d\tau\,\right)^2dx\\
\leq&\,4\,\sum_{m=1}^{\ssy M}\int_{\ssy\Delta_m}
\left[\,\|v(\tau,\cdot)\|_{\ssy 0,D}^2 +(\tau-\tau_{m-1})^2\,
\|v_{\tau}(\tau,\cdot)\|_{\ssy 0,D}^2\,\right]\;d\tau\\
\leq&\,4\,\int_0^{\ssy T}\left(\,\|v(\tau,\cdot)\|_{\ssy 0,D}^2
+\tau^2\,\|v_{\tau}(\tau,\cdot)\|_{\ssy 0,D}^2\,\right)\;d\tau,
\end{split}
\end{equation*}
which, along with \eqref{Reggo3} (taking $(\beta,\ell,p)=(0,0,0)$
and $(\beta,\ell,p)=(2,1,0)$), yields
\begin{equation}\label{May2015_24}
2\,\dtau\,\sum_{m=1}^{\ssy M}\|v^m\|_{\ssy 0,D}^2\leq\,C\,\|v_0\|_{\ssy {\bfdot H}^{-1}}^2.
\end{equation}
Observing that  ${\mathcal E}^{\half}=\half\,{\mathcal E}^1$, we have
\begin{equation*}
\begin{split}
\dtau\,\sum_{m=1}^{\ssy M}\|{\mathcal E}^{m-\half}\|_{\ssy 0,D}^2
\leq&\,2\,\left(\,
\dtau\,\|\VV^1\|_{\ssy 0,D}^2
+\dtau\,\sum_{m=2}^{\ssy M}\|\VV^{m-\half}\|_{\ssy 0,D}^2\,\right)\\
&\,+2\, \left(\,
\dtau\,\|v^1\|_{\ssy 0,D}^2+\dtau\,\sum_{m=2}^{\ssy M}\|v^{m-\half}\|_{\ssy 0,D}^2\,\right),\\
\end{split}
\end{equation*}
which, after using \eqref{May2015_23}, \eqref{Museum_1}
and \eqref{May2015_24}, yields that \eqref{Ydaspis900} holds for $\theta=0$.
%
%
%
\par
Hence, the estimate \eqref{Ydaspis900} follows by interpolation.
\end{proof}
%
%
\par
Next, we establish a discrete in time $L^2_t(L^2_x)$ estimate of the nodal error.
%
%
\begin{proposition}\label{Propo_Boom1}
Let $(\VV^m)_{m=0}^{\ssy M}$ be the modified Crank-Nicolson time-discrete
approximations of the solution $v$ to the problem \eqref{Det_Parab}
defined by \eqref{CNDet1}--\eqref{CNDet2}.
Then, there exists a constant $C>0$, independent of $\dtau$, such that
\begin{equation}\label{September2015_0}
\left(\,\dtau\,\sum_{m=1}^{\ssy M}\|\VV^m-v^m\|_{\ssy 0,D}^2\,\right)^{\ssy 1/2}
\leq\,C\,\dtau^{\frac{\delta}{2}}\,\|v_0\|_{\ssy{\bfdot H}^{\delta-1}}
\quad\forall\,\delta\in[0,1],\quad\forall\,v_0\in{\bfdot H}^1(D),
\end{equation}
 where $v^{\ell}:=v(\tau_{\ell},\cdot)$ for $\ell=0,\dots,M$.
\end{proposition}
%
%
\begin{proof}
We will arrive at the error bound \eqref{September2015_0} by interpolation after
proving it for $\delta=1$ and $\delta=0$  (cf. Proposition~\ref{DetPropo1}).
In both cases, the error estimation is based on the following bound
\begin{equation}\label{Hadji0}
\left(\dtau\sum_{m=1}^{\ssy M}\|\VV^m-v^m\|_{\ssy 0,D}^2\right)^{\ssy 1/2}\leq
S_1+S_2+S_3
\end{equation}
where
\begin{equation*}
\begin{split}
S_1:=&\,\left(\dtau\sum_{m=2}^{\ssy M}\|\VV^m-\VV^{m-\half}\|_{\ssy 0,D}^2\right)^{\ssy 1/2},\\
S_2:=&\,\left(\dtau\,\|\VV^1-v^1\|_{\ssy 0,D}^2
+\dtau\sum_{m=2}^{\ssy M}\|\VV^{m-\half}-v^{m-\half}\|_{\ssy 0,D}^2\right)^{\ssy 1/2},\\
S_3:=&\,\left(\dtau\sum_{m=2}^{\ssy M}\|v^{m-\half}-v^m\|_{\ssy 0,D}^2\right)^{\ssy 1/2}.\\
\end{split}
\end{equation*}
In the sequel,  we will use the symbol
$C$ to denote a generic constant that is independent of $\dtau$ and may change value from
one line to the other.
\par
Taking the $L^2(D)-$inner product of both sides of \eqref{CNDet2} with $(\VV^m-\VV^{m-1})$
and then integrating by parts, we easily arrive at
\begin{equation}\label{Hadji1}
\|\VV^m-\VV^{m-1}\|^2_{\ssy 0,D}+\tfrac{\dtau}{2}\,\left(\,|\VV^m|_{\ssy 1,D}^2
-|\VV^{m-1}|_{\ssy 1,D}^2\,\right)=0,\quad m=2,\dots,M.
\end{equation}
After summing both sides of \eqref{Hadji1} with respect to $m$ from $2$ up to $M$, we obtain
\begin{equation*}
\dtau\sum_{m=2}^{\ssy M}\|\VV^m-\VV^{m-1}\|^2_{\ssy 0,D}
+\tfrac{\dtau^2}{2}\,\left(\,|\VV^{\ssy M}|_{\ssy 1,D}^2
-|\VV^1|_{\ssy 1,D}^2\,\right)=0,
\end{equation*}
which yields
\begin{equation}\label{Hadji2}
\dtau\sum_{m=2}^{\ssy M}\|\VV^m-\VV^{m-1}\|^2_{\ssy 0,D}
\leq\tfrac{\dtau^2}{2}\,|\VV^1|_{\ssy 1,D}^2.
\end{equation}
Taking the $L^2(D)-$inner product of both sides of \eqref{CNDet11} with $\VV^1$,
and then integrating by parts and using \eqref{innerproduct}, we obtain:
$\|\VV^1\|_{\ssy 0,D}^2-\|\VV^0\|_{\ssy 0,D}^2+\dtau\,|\VV^1|_{\ssy 1,D}^2\leq0$,
which yields
\begin{equation}\label{Hadji3}
\dtau\,|\VV^1|_{\ssy 1,D}^2\leq\|v_0\|_{\ssy 0,D}^2.
\end{equation}
Thus, combining \eqref{Hadji2} and \eqref{Hadji3}, we get
\begin{equation}\label{Hadji4}
\begin{split}
S_1=&\,\tfrac{1}{2}\,
\left(\dtau\sum_{m=2}^{\ssy M}\|\VV^m-\VV^{m-1}\|_{\ssy 0,D}^2\right)^{\half}\\
\leq&\,\tfrac{1}{2\sqrt{2}}\,\dtau\,|\VV^1|_{\ssy 1,D}\\
\leq&\,\sqrt{\dtau}\,\|v_0\|_{\ssy 0,D}.\\
\end{split}
\end{equation}
Also, we observe that the estimate \eqref{Ydaspis900}, for $\theta=\frac{1}{2}$, yields
\begin{equation}\label{Hadji44}
S_2\leq\,C\,\sqrt{\dtau}\,\|v_0\|_{\ssy 0,D}.
\end{equation}
Finally, using \eqref{Reggo0} (with $\ell=1$, $p=0$, $q=0$), we obtain
\begin{equation}\label{Hadji45}
\begin{split}
S_3\leq&\,\left(\dtau\sum_{m=2}^{\ssy M}
\left\|\int_{\ssy\Delta_m}\partial_{\tau}v(\tau,\cdot)\;d\tau\right\|_{\ssy 0,D}^2
\right)^{\frac{1}{2}}\\
\leq&\,\left(\dtau^2\int_{\dtau}^{\ssy T}\|\partial_{\tau}v(\tau,\cdot)\|^2_{\ssy 0,D}\;d\tau
\right)^{\frac{1}{2}}\\
\leq&\,C\,\left(\dtau^2\int_{\dtau}^{\ssy T}\tau^{-2}\,\|v_0\|^2_{\ssy 0,D}\;d\tau
\right)^{\frac{1}{2}}\\
\leq&\,C\,\dtau\,\|v_0\|_{\ssy 0,D}\,\left(\tfrac{1}{\dtau}-\tfrac{1}{T}\right)^{\frac{1}{2}}\\
\leq&\,C\,\sqrt{\dtau}\,\|v_0\|_{\ssy 0,D}.\\
\end{split}
\end{equation}
Thus, from \eqref{Hadji0}, \eqref{Hadji4}, \eqref{Hadji44} and \eqref{Hadji45} we conclude
\eqref{September2015_0} for $\delta=1$.
\par
Taking again the $L^2(D)-$inner product of both sides of \eqref{CNDet11} with $\VV^1$
and then integrating by parts and using \eqref{Poincare},
\eqref{minus_equiv} and \eqref{H_equiv} along with the
arithmetic mean inequality,  we obtain
\begin{equation*}\label{Hadji5}
\begin{split}
\|\VV^1\|_{\ssy 0,D}^2+\tfrac{\dtau}{2}\,|\VV^1|_{\ssy 1,D}^2=&\,(v_0,\VV^1)_{\ssy 0,D}\\
\leq&\,\|v_0\|_{\ssy -1,D}\,\|\VV^1\|_{\ssy 1,D}\\
\leq&\,C\,\|v_0\|_{\ssy{\bfdot H}^{-1}}\,|\VV^1|_{\ssy 1,D}\\
\leq&\,C\,\dtau^{-1}\,\|v_0\|^2_{\ssy{\bfdot H}^{-1}}
+\tfrac{\dtau}{4}\,|\VV^1|_{\ssy 1,D}^2,
\end{split}
\end{equation*}
%
%
%
which yields that
\begin{equation}\label{Hadji6}
\dtau\,|\VV^1|_{\ssy 1,D}\leq\,C\,\|v_0\|_{\ssy{{\bfdot H}^{-1}}}.
\end{equation}
Thus, combining \eqref{Hadji2} and \eqref{Hadji6}, we conclude that
\begin{equation}\label{Hadji8}
\begin{split}
S_1=&\,\tfrac{1}{2}\,\left(\dtau
\sum_{m=2}^{\ssy M}\|\VV^m-\VV^{m-1}\|_{\ssy 0,D}^2\right)^{\half}\\
\leq&\,\tfrac{1}{2\sqrt{2}}\,\dtau\,|\VV^1|_{\ssy 1,D}\\
\leq&\,C\,\|v_0\|_{\ssy{{\bfdot H}^{-1}}}.
\end{split}
\end{equation}
Also, the estimate \eqref{Ydaspis900}, for $\theta=0$, yields
\begin{equation}\label{Hadji11}
S_2\leq\,C\,\|v_0\|_{\ssy{{\bfdot H}^{-1}}}.
\end{equation}
Using the Cauchy-Schwarz inequality and \eqref{May2015_24}, we have
\begin{equation}\label{Hadji12}
\begin{split}
S_3=&\,\tfrac{1}{2}
\,\left(\dtau\sum_{m=2}^{\ssy M}\|v^m-v^{m-1}\|_{\ssy 0,D}^2\right)^{\half}\\
\leq&\,\tfrac{\sqrt{2}}{2}
\,\left(2\,\dtau\sum_{m=1}^{\ssy M}\|v^m\|_{\ssy 0,D}^2\right)^{\half}\\
\leq&\,C\,\|v_0\|_{\ssy {\bfdot H}^{-1}}.\\
\end{split}
\end{equation}
Thus, from \eqref{Hadji0}, \eqref{Hadji8}, \eqref{Hadji11} and \eqref{Hadji12} we conclude
\eqref{September2015_0} for $\delta=0$.
\end{proof}
%
%
%
%
%
%
\subsection{Fully-Discrete Approximations}\label{Section42}
%
%
In this section we construct and analyze finite element approximations,
$(\VV_h^m)_{m=0}^{\ssy M}$,  of the modified Crank-Nicolson
time-discrete approximations defined in Section~\ref{Section41}.
The method begins by setting
\begin{equation}\label{CNFD1}
\VV_h^0:=P_hv_0
\end{equation}
and finding $\VV_h^1\in\fem$ such that
\begin{equation}\label{CNFD12}
\VV_h^1-\VV_h^0+\tfrac{\dtau}{2}\,\Delta_h\VV_h^1 =0.
\end{equation}
Then, for $m=2,\dots,M$, it specifies $\VV_h^m\in\fem$ such
that
\begin{equation}\label{CNFD2}
\VV_h^m-\VV_h^{m-1}+\dtau\,\Delta_h\VV_h^{m-\frac{1}{2}}=0.
\end{equation}
\par
First,  we show a discrete in time $L^2_t(L^2_x)$ a priori estimate of time averages
of the nodal error between the modified Crank-Nicolson time-discrete approximations
presented in the previous section and
the modified Crank-Nicolson fully-discrete approximations defined above.
%
%
\begin{proposition}\label{Aygo_Kokora}
Let $(\VV^m)_{m=0}^{\ssy M}$ be the Crank-Nicolson time-discrete
approxi\-mations defined by \eqref{CNDet1}--\eqref{CNDet2}
and $(\VV_h^m)_{m=0}^{\ssy M}$ be the modified Crank-Nicolson
fully-discrete approximations defined by \eqref{CNFD1}--\eqref{CNFD2}.
Then, there exists a constant $C>0$,
independent of $h$ and $\dtau$, such that
\begin{equation}\label{August2017_1}
\left(\,\dtau\,\|\VV^{1}-\VV_h^{1}\|^2_{\ssy 0,D}
+\dtau\sum_{m=2}^{\ssy M}
\|\VV^{m-\frac{1}{2}}-\VV_h^{m-\frac{1}{2}}\|^2_{\ssy
0,D}\,\right)^{\ssy 1/2}\leq \,C\,\,h^{2\theta}
\,\,\|v_0\|_{\ssy {\bfdot H}^{2\theta-1}}
\end{equation}
for all $\theta\in[0,1]$ and $v_0\in {\bfdot H}^1(D)$.
\end{proposition}
%
%
%
%
%
%
%
\begin{proof}
We will get the error estimate \eqref{August2017_1} by interpolation after proving it for $\theta=1$ and
$\theta=0$ (cf. \cite{KZ2010}). In the sequel,  we will use the symbol
$C$ to denote a generic constant that is independent of $\dtau$ and $h$,
and may change value from one line to the other.
\par
Letting ${\sf \Theta}^{\ell}:=\VV^{\ell}-\VV_h^{\ell}$ for $\ell=0,\dots,M$, we use
\eqref{CNDet11}, \eqref{CNFD12}, \eqref{CNDet2} and \eqref{CNFD2},
to arrive at the following error equations:
\begin{gather}
T_{\ssy E,h}({\sf\Theta}^1-{\sf\Theta}^{0})=\tfrac{\dtau}{2}\,{\sf\Theta}^1
+\tfrac{\dtau}{2}\,\xi_1,\label{SaintMinas_1a}\\
T_{\ssy E,h}({\sf\Theta}^m-{\sf\Theta}^{m-1})=\dtau\,{\sf\Theta}^{m-\frac{1}{2}}
+\dtau\,\xi_m,\quad m=2,\dots,M,\label{SaintMinas_1}
\end{gather}
where
\begin{gather}
\xi_1:=(T_{\ssy E,h}-T_{\ssy E})\partial^2\VV^1\label{xi_def_1},\\
\xi_{\ell}:=(T_{\ssy E,h}-T_{\ssy E})\partial^2\VV^{\ell-\frac{1}{2}},
\quad\ell=2,\dots,M.\label{xi_def_2}
\end{gather}
Taking the $L^2(D)-$inner product of both sides of
\eqref{SaintMinas_1} with ${\sf \Theta}^{m-\frac{1}{2}}$ and then using \eqref{Frako1},
the Cauchy-Schwarz inequality along with the arithmetic mean inequality, we obtain
\begin{equation*}
|T_{\ssy E,h}{\sf\Theta}^m|_{\ssy 1,D}^2
-|T_{\ssy B,h}{\sf\Theta}^{m-1}|_{\ssy 1,D}^2
+\dtau\,\|{\sf\Theta}^{m-\frac{1}{2}}\|_{\ssy
0,D}^2\leq\dtau\,\|\xi_m\|_{\ssy 0,D}^2,\quad m=2,\dots,M.
\end{equation*}
After summing with respect to $m$ from $2$ up to $M$, the relation above yields
%
%
%
\begin{equation}\label{citah3}
\dtau\,\|{\sf\Theta}^1\|_{\ssy 0,D}^2
+\dtau\,\sum_{m=2}^{\ssy M}\|{\sf\Theta}^{m-\frac{1}{2}}\|_{\ssy 0,D}^2
\leq\,|T_{\ssy E,h}{\sf\Theta}^1|_{\ssy 1,D}^2
+\dtau\,\|{\sf\Theta}^1\|^2_{\ssy 0,D}
+\dtau\sum_{m=2}^{\ssy M}\left\|\xi_m\right\|_{\ssy 0,D}^2.
\end{equation}
Observing that $T_{\ssy E,h}{\sf\Theta}^0=0$, we
take the $L^2(D)-$inner product of both sides of
\eqref{SaintMinas_1a} with ${\sf \Theta}^{1}$ and then
use the Cauchy-Schwarz inequality along with the
arithmetic mean inequality, to get
\begin{equation}\label{July2015_1}
|T_{\ssy E,h}{\sf\Theta^1}|_{\ssy 1,D}^2+\tfrac{\dtau}{4}\,
\|{\sf\Theta^1}\|^2_{\ssy 0,D}\leq\,\tfrac{\dtau}{4}\,\|\xi_1\|_{\ssy 0,D}^2.
\end{equation}
Thus, using \eqref{citah3}, \eqref{July2015_1},
\eqref{xi_def_1}, \eqref{xi_def_2} and \eqref{pin1},
we easily conclude that
\begin{equation}\label{July2015_2}
\begin{split}
\dtau\,\|{\sf\Theta}^1\|_{\ssy 0,D}^2
&+\dtau\,\sum_{m=2}^{\ssy M}\|{\sf\Theta}^{m-\frac{1}{2}}\|_{\ssy 0,D}^2
\leq\,\dtau\,\|\xi_1\|_{\ssy 0,D}^2
+\dtau\,\sum_{m=2}^{\ssy M}\left\|\xi_m\right\|_{\ssy 0,D}^2\\
\leq&\,C\,h^{4}\,\left(\,\dtau\,|\VV^1|^2_{\ssy 2,D}
+\dtau\,\sum_{m=2}^{\ssy M}
|\VV^{m-\frac{1}{2}}|_{\ssy 2,D}^2\right).
\end{split}
\end{equation}
Taking the $L^2(D)-$inner product of \eqref{CNDet2} with
$\partial^2\VV^{m-\frac{1}{2}}$,
and then integrating by parts and summing with respect to $m$,
from $2$ up to $M$, it follows that
\begin{equation*}
|\VV^{\ssy M}|_{\ssy 1,D}^2-|\VV^1|_{\ssy 1,D}^2
+2\,\dtau\,\sum_{m=2}^{\ssy M}
|\VV^{m-\frac{1}{2}}|_{\ssy 2,D}^2=0
\end{equation*}
which yields
\begin{equation}\label{citah6}
\dtau\,|\VV^1|_{\ssy 2,D}^2
+\sum_{m=2}^{\ssy M}
\dtau\,|\VV^{m-\frac{1}{2}}|_{\ssy 2,D}^2
\leq\,\tfrac{1}{2}\,|\VV^1|_{\ssy 1,D}^2
+\dtau\,|\VV^1|_{\ssy 2,D}^2.
\end{equation}
Now, take the $L^2(D)-$inner product of \eqref{CNDet11}
with $\partial^2\VV^1$, and then
integrate by parts and use \eqref{innerproduct} to get
\begin{equation*}
|\VV^1|^2_{\ssy 1,D}-|\VV^0|_{\ssy 1,D}^2
+\dtau\,|\VV^1|_{\ssy 2,D}^2\leq0,
\end{equation*}
which, along with \eqref{H_equiv}, yields
\begin{equation}\label{citah6a}
|\VV^1|_{\ssy 1,D}^2
+\dtau\,|\VV^1|_{\ssy 2,D}^2\leq\,\|v_0\|_{\ssy {\bfdot H}^1}^2.
\end{equation}
Thus, combining \eqref{July2015_2}, \eqref{citah6} and \eqref{citah6a},
we obtain \eqref{August2017_1} for $\theta=1$.
%
%
%
\par
From \eqref{CNFD12} and \eqref{CNFD2}, it follows that
\begin{equation}\label{Frago_0}
-T_{\ssy E,h}(\VV_h^1-\VV_h^0)+\tfrac{\dtau}{2}\,\VV_h^1=0,
\end{equation}
\begin{equation}\label{Frago_1}
-T_{\ssy E,h}(\VV_h^m-\VV_h^{m-1})+\dtau\,\VV_h^{m-\frac{1}{2}}=0,
\quad m=2,\dots,M.
\end{equation}
Taking the $L^2(D)-$inner product
of \eqref{Frago_1} with $\VV_h^{m-\frac{1}{2}}$ and using \eqref{Frako1},
we have
\begin{equation*}
|T_{\ssy E,h}\VV_h^m|_{\ssy 1,D}^2
-|T_{\ssy E,h}\VV_h^{m-1}|_{\ssy 1,D}^2
+2\,\dtau\,\|\VV_h^{m-\frac{1}{2}}\|_{\ssy 0,D}^2=0, \quad
m=2,\dots,M,
\end{equation*}
which, after summing with respect to $m$ from $2$ up to $M$, yields
\begin{equation}\label{Frago_2}
\dtau\,\|\VV_h^1\|^2_{\ssy 0,D}
+\dtau\sum_{m=2}^{\ssy M}\|\VV_h^{m-\frac{1}{2}}\|_{\ssy 0,D}^2
\leq\tfrac{1}{2}\,|T_{\ssy E,h}\VV_h^1|_{\ssy 1,D}^2
+\dtau\,\|\VV_h^1\|^2_{\ssy 0,D}.
\end{equation}
Now, take the $L^2(D)-$inner product
of \eqref{Frago_0} with $\VV_h^1$ and use \eqref{innerproduct}
and \eqref{CNFD1}, to have
\begin{equation}\label{Frago_3}
\begin{split}
|T_{\ssy E,h}\VV_h^1|_{\ssy 1,D}^2
+\dtau\,\|\VV_h^{1}\|_{\ssy 0,D}^2\leq&\,
|T_{\ssy E,h}P_hv_0|_{\ssy 1,D}^2\\
\leq&\,|T_{\ssy E,h}v_0|_{\ssy 1,D}^2.\\
\end{split}
\end{equation}
Combining \eqref{Frago_2}, \eqref{Frago_3}, \eqref{DELSTAB}
and \eqref{minus_equiv}, we obtain
\begin{equation}\label{Frago_4}
\left(\dtau\,\|\VV_h^1\|^2_{\ssy 0,D}
+\dtau\,\sum_{m=2}^{\ssy M}
\,\|\VV_h^{m-\frac{1}{2}}\|_{\ssy 0,D}^2\right)^{\ssy 1/2}
\leq\,C\,\|v_0\|_{\ssy{\bfdot H}^{-1}}.
\end{equation}
Finally, combine \eqref{Frago_4} with
\eqref{May2015_23} to get \eqref{August2017_1} for $\theta=0$.
\end{proof}
%
%
\par
Next,  we derive a discrete in time $L^2_t(L^2_x)$ a priori estimate
of the nodal error between the modified Crank-Nicolson time-discrete approximations and
the modified Crank-Nicolson fully-discrete approximations.
%
%
\begin{proposition}\label{X_Aygo_Kokora}
Let
$(\VV^m)_{m=0}^{\ssy M}$ be the modified Crank-Nicolson time-discrete
approximations defined by \eqref{CNDet1}--\eqref{CNDet2},
and $(\VV_h^m)_{m=0}^{\ssy M}$ be the modified Crank-Nicolson finite element
approximations specified by \eqref{CNFD1}--\eqref{CNFD2}.
Then, there exists a constant $C>0$,
independent of $h$ and $\dtau$, such that
\begin{equation}\label{X_tiger_river1}
\left(\,\dtau\sum_{m=1}^{\ssy M}
\|\VV^{m}-\VV_h^{m}\|^2_{\ssy 0,D}\,\right)^{\ssy 1/2}
\leq \,C\,\left(\,\dtau^{\frac{\delta}{2}}
\,\|v_0\|_{\ssy {\bfdot H}^{\delta-1}}
+h^{2\theta}
\,\,\|v_0\|_{\ssy {\bfdot H}^{2\theta-1}}\,\right)
\end{equation}
for all $\delta$, $\theta\in[0,1]$ and $v_0\in {\bfdot H}^2(D)$.
\end{proposition}
%
%
%
%
%
\begin{proof}
The proof is based on the estimation of the terms
in the right hand side of the following triangle inequality:
\begin{equation}\label{X_Hadji0}
\left(\dtau\sum_{m=1}^{\ssy M}\|\VV^m-\VV_h^m\|_{\ssy 0,D}^2
\right)^{\ssy 1/2}\leq
{\mathcal S}_1+{\mathcal S}_2+{\mathcal S}_3
\end{equation}
where
\begin{equation*}
\begin{split}
{\mathcal S}_1:=&\,\left(\dtau\sum_{m=2}^{\ssy M}\|\VV^m-\VV^{m-\half}\|_{\ssy 0,D}^2\right)^{\ssy 1/2},\\
{\mathcal S}_2:=&\,\left(\dtau\,\|\VV^1-\VV_h^1\|_{\ssy 0,D}^2
+\dtau\sum_{m=2}^{\ssy M}\|\VV^{m-\half}-\VV_h^{m-\half}\|_{\ssy 0,D}^2\right)^{\ssy 1/2},\\
{\mathcal S}_3:=&\,\left(\dtau\sum_{m=2}^{\ssy M}\|\VV_h^{m-\half}-\VV_h^m\|_{\ssy 0,D}^2\right)^{\ssy 1/2}.\\
\end{split}
\end{equation*}
In the sequel,  we will use the symbol $C$ to denote a generic constant that is
independent of $\dtau$ and may change value from one line to the other.
\par
Taking the $L^2(D)-$inner product of both sides of \eqref{CNFD2}
with $(\VV_h^m-\VV_h^{m-1})$, we have
\begin{equation}\label{X_Hadji1}
\|\VV_h^m-\VV_h^{m-1}\|^2_{\ssy 0,D}
+\tfrac{\dtau}{2}\,\left(\,|\VV_h^m|_{\ssy 1,D}^2
-|\VV_h^{m-1}|_{\ssy 1,D}^2\,\right)=0,\quad m=2,\dots,M.
\end{equation}
After summing both sides of \eqref{X_Hadji1} with respect to $m$ from $2$ up to $M$, we obtain
\begin{equation*}
\dtau\sum_{m=2}^{\ssy M}\|\VV_h^m-\VV_h^{m-1}\|^2_{\ssy 0,D}
+\tfrac{\dtau^2}{2}\,\left(\,|\VV_h^{\ssy M}|_{\ssy 1,D}^2
-|\VV_h^{1}|_{\ssy 1,D}^2\,\right)=0,
\end{equation*}
which yields
\begin{equation}\label{X_Hadji2}
\dtau\sum_{m=2}^{\ssy M}\|\VV_h^m-\VV_h^{m-1}\|^2_{\ssy 0,D}
\leq\tfrac{\dtau^2}{2}\,|\VV_h^{1}|_{\ssy 1,D}^2.
\end{equation}
Taking the $L^2(D)-$inner product of both sides of \eqref{CNFD12} with $\VV_h^1$
and then using \eqref{innerproduct}, we obtain
\begin{equation*}
\|\VV_h^1\|_{\ssy 0,D}^2-\|\VV_h^0\|_{\ssy 0,D}^2
+\dtau\,|\VV_h^1|_{\ssy 1,D}^2\leq0
\end{equation*}
from which we conclude that
\begin{equation}\label{X_Hadji3}
\dtau\,|\VV_h^1|_{\ssy 1,D}^2\leq\|v_0\|_{\ssy 0,D}^2.
\end{equation}
Thus, combining \eqref{X_Hadji2} and \eqref{X_Hadji3} we have
\begin{equation}\label{X_Hadji4}
\begin{split}
{\mathcal S}_3=&\,\tfrac{1}{2}\,
\left(\dtau\sum_{m=2}^{\ssy M}\|\VV_h^m-\VV_h^{m-1}\|_{\ssy 0,D}^2\right)^{\ssy 1/2}\\
\leq&\,\tfrac{1}{2\sqrt{2}}\,\dtau\,|\VV^1_h|_{\ssy 1,D}\\
\leq&\,\dtau^{\half}\,\|v_0\|_{\ssy 0,D}.\\
\end{split}
\end{equation}
Taking again the $L^2(D)-$inner product of both sides of \eqref{CNFD12} with $\VV_h^1$
and then using \eqref{minus_equiv} and \eqref{Poincare} along with the
arithmetic mean inequality,  we obtain
\begin{equation*}\label{X_Hadji5}
\begin{split}
\|\VV_h^1\|_{\ssy 0,D}^2+\tfrac{\dtau}{2}
\,|\VV_h^1|_{\ssy 1,D}^2=&\,(P_hv_0,\VV_h^1)_{\ssy 0,D}\\
=&\,(v_0,\VV_h^1)_{\ssy 0,D}\\
\leq&\,\|v_0\|_{\ssy -1,D}\,\|\VV_h^1\|_{\ssy 1,D}\\
\leq&\,C\,\|v_0\|_{\ssy{\bfdot H}^{-1}}\,|\VV_h^1|_{\ssy 1,D}\\
\leq&\,C\,\dtau^{-1}\|v_0\|^2_{\ssy{\bfdot H}^{-1}}
+\tfrac{\dtau}{4}\,|\VV_h^1|_{\ssy 1,D}^2,
\end{split}
\end{equation*}
which yields that
\begin{equation}\label{X_Hadji6}
\dtau^2\,|\VV_h^1|_{\ssy 1,D}^2
\leq\,C\,\|v_0\|_{\ssy{{\bfdot H}^{-1}}}^2.
\end{equation}
Thus, combining \eqref{X_Hadji2} and \eqref{X_Hadji6}, we conclude that
\begin{equation}\label{X_Hadji8}
\begin{split}
{\mathcal S}_3
=&\,\tfrac{1}{2}\,
\left(\dtau\sum_{m=2}^{\ssy M}\|\VV_h^m-\VV_h^{m-1}\|_{\ssy 0,D}^2\right)^{\ssy 1/2}\\
\leq&\,\tfrac{1}{2\sqrt{2}}\,\dtau\,|\VV^1_h|_{\ssy 0,D}\\
\leq&\,C\,\|v_0\|_{\ssy{{\bfdot H}^{-1}}}.
\end{split}
\end{equation}
Also, from \eqref{Hadji4} and \eqref{Hadji8}, we have
\begin{equation}\label{XX_Hadji4}
{\mathcal S}_1\leq\,C\,\dtau^{\half}\,\|v_0\|_{\ssy 0,D}
\end{equation}
and
\begin{equation}\label{XX_Hadji44}
{\mathcal S}_1\leq\,C\,\|v_0\|_{\ssy{{\bfdot H}^{-1}}}.
\end{equation}
By interpolation, from \eqref{X_Hadji4}, \eqref{XX_Hadji4}, \eqref{X_Hadji8}
and \eqref{XX_Hadji44}, we conclude that
\begin{equation}\label{XXX_1}
{\mathcal S}_2+{\mathcal S}_3\leq\,C\,\dtau^{\frac{\delta}{2}}
\,\|w_0\|_{\ssy{{\bfdot H}^{\delta-1}}}
\quad\forall\,\delta\in[0,1].
\end{equation}
Finally, the estimate \eqref{August2017_1} reads
\begin{equation}\label{XXX_2}
{\mathcal S}_2\leq\,C\,h^{2\theta}\,\|v_0\|_{\ssy{\bfdot H}^{2\theta-1}}
\quad\forall\,\theta\in[0,1].
\end{equation}
Thus, \eqref{X_tiger_river1} follows as a simple consequence of
\eqref{X_Hadji0}, \eqref{XXX_1} and \eqref{XXX_2}.
\end{proof}
%
%
%
%
\section{Convergence Analysis of the Crank-Nicolson finite element method}\label{Section5}
%
%
In this section, we focus on the derivation of an estimate of the approximation error of
the Crank-Nicolson finite element method introduced in Section~\ref{kefalaiaki}. For that,
we will use, as a comparison tool (cf. \cite{KZ2010}, \cite{Zouraris2017}), the corresponding
Crank-Nicolson time-discrete approximations of ${\widehat u}$, which are defined first by setting
\begin{equation}\label{CN_S1}
U^0:=0
\end{equation}
and then, for $m=1,\dots,M$, by specifying $U^m\in{\bfdot H}^2(D)$ such that
\begin{equation}\label{CN_S2}
U^m-U^{m-1}=\dtau\,\partial^2 U^{m-\frac{1}{2}}
+\int_{\ssy\Delta_m}{\mathcal W}\,ds\quad\text{\rm a.s.}.
\end{equation}
Thus, we split the {\sl discretization error} of the Crank Nicolson
finite element method as follows
\begin{equation}\label{split_split}
\max_{0\leq{m}\leq{\ssy M}}\left({\mathbb E}\left[\,
\|\uu^m-U_h^m\|_{\ssy 0,D}^2\,\right]\right)^{\ssy 1/2}
\leq\max_{1\leq{m}\leq{\ssy M}}{\sf E}_{\ssy\sf TDR}^m
+\max_{1\leq{m}\leq{\ssy M}}{\sf E}^m_{\ssy\sf SDR},
\end{equation}
where $\uu^m:=\uu(\tau_m,\cdot)$,
${\sf E}^m_{\ssy\sf TDR}:=
\left({\mathbb E}\left[\|\uu^m-U^m\|_{\ssy 0,D}^2 \right]\right)^{\ssy 1/2}$
is the {\sl time discretization error}  at $\tau_m$, and
${\sf E}_{\ssy\sf SDR}^m:=
\left({\mathbb E}\left[\|U^m-U^m_h\|_{\ssy 0,D}^2 \right]\right)^{\ssy 1/2}$
is the {\sl space discretization error} at $\tau_m$.
\par
In the sequel, we estimate the above defined type of error
using a discrete Duhamel principle technique along with the
convergence results for the modified Crank-Nicolson method
obtained in Section~\ref{Section4} (cf. \cite{YubinY04}, 
\cite{YubinY05}, \cite{BinLi}, \cite{KZ2010}, \cite{Zouraris2017}).
%
%
\subsection{Estimating the time discretization error}
%
%
%
\begin{proposition}\label{TimeDiscreteErr1}
Let $(U^m)_{m=0}^{\ssy M}$ be the Crank-Nicolson
time discrete approximations specified
by \eqref{CN_S1}-\eqref{CN_S2}. Then, there exists
a constant $C_{\ssy\sf TDR}>0$, independent of $\dt$, $\dx$ and $\dtau$,
such that
\begin{equation}\label{Dinner0}
\max_{1\leq m \leq {\ssy M}} {\sf E}_{\ssy\sf TDR}^m \leq\,C_{\ssy\sf TDR}
\,\epsilon^{-\half}\,\dtau^{\frac{1}{4}-\epsilon}
\quad\forall\,\epsilon\in\left(0,\tfrac{1}{4}\right].
\end{equation}
\end{proposition}
%
%
%
%
%
\begin{proof}
Let ${\sf I}:L^2(D)\to L^2(D)$ be the identity operator,
${\sf Y}:H^2(D)\rightarrow L^2(D)$ be defined by
${\sf Y}:={\sf I}+\tfrac{\dtau}{2}\,\partial^2$
and $\Lambda:L^2(D)\to{\bfdot H}^2(D)$ be the inverse
elliptic operator $\Lambda:=({\sf I}-\tfrac{\dtau}{2}\,\partial^2)^{-1}$.
Then, for $m=1,\dots,M$, we define an operator
${\mathcal Q}^m:L^2(D)\to{\bfdot H}^2(D)$ by
${\mathcal Q}^m:=(\Lambda\circ{\sf Y})^{m-1}\circ\Lambda$,
which has a Green's function ${\mathcal G}_{\ssy{\mathcal Q}^m}$
given by
%
${\mathcal G}_{\ssy{\mathcal Q}^m}(x,y):=
\sum_{\kappa=1}^{\infty}\tfrac{(1-\frac{\dtau}{2}\lambda_{\kappa}^2)^{m-1}}
{(1+\frac{\dtau}{2}\lambda_{\kappa}^2)^{m}}\,\varepsilon_{\kappa}(x)
\,\varepsilon_{\kappa}(y)\quad\forall\,x,y\in{\overline D}$.
%
%
Also, for given $v_0\in{\bfdot H}^1(D)$, let $({\mathcal S}_{\ssy\dtau}^m(v_0))_{m=0}^{\ssy M}$
be the time-discrete approximations of the solution to the deterministic problem \eqref{Det_Parab},
defined by \eqref{CNDet1}--\eqref{CNDet2}.  Then, using a simple induction argument, we conclude that
${\mathcal S}_{\ssy\dtau}^m(v_0)={\mathcal Q}^{m}(v_0)$ for $m=1,\dots,M$.
To simplify the notation, we set
$G_m(\tau;x,y):={\mathcal X}_{(0,\tau_m)}(\tau)
\,\,G_{\tau_m-\tau}(x,y)$ for $m=1,\dots,M$.
Also, we will use the symbol $C$ to denote a generic constant that is independent
of $\dt$, $\dx$ and $\dtau$, and may change value from one line to the other.
\par
An induction argument applied to \eqref{CN_S2}, yields
\begin{equation}\label{Dinner1}
\begin{split}
U^m(x)=&\,\sum_{\ell=1}^{\ssy m}
\int_{\ssy\Delta_{\ell}}{\mathcal S}_{\ssy\dtau}^{m-\ell+1}
\left({\mathcal W}(\tau,x)\right)\,d\tau\\
=&\,\int_0^{\ssy T}\int_{\ssy D}
{\mathcal K}^m(\tau;x,y)\,{\mathcal W}(\tau,y)\,dy d\tau
\quad\forall\,x\in{\overline D},\quad m=1,\dots,M,\\
\end{split}
\end{equation}
with
${\mathcal K}^m(\tau;x,y):=\sum_{\ell=1}^{m}
{\mathcal X}_{\ssy\Delta_{\ell}}(\tau) \,{\mathcal G}_{{\mathcal Q}^{m-\ell+1}}(x,y)$ for
$x,y\in{\overline D}$, $\tau\in[0,T]$.
%
%
%
Then, using \eqref{Dinner1}, \eqref{HatUform}, \eqref{WNEQ2},
\eqref{Ito_Isom}, \eqref{HSxar} and \eqref{peykos1}, we obtain
\begin{equation*}
\begin{split}
{\sf E}_{\ssy\sf TDR}^m=&\,\left(\,{\mathbb E}\left[\,\int_{\ssy D}\left(
\int_0^{\ssy T}\!\!\!\int_{\ssy D}
\,\left(\,{\mathcal K}_{m}-G_{m}\right)(\tau;x,y)\,
{\mathcal W}(\tau,y)\,dyd\tau\,\right)^2dx\,\right]\,\right)^{\ssy 1/2}\\
\leq&\,\left(\,\int_0^{\tau_{m}}\left(\int_{\ssy D}\!\int_{\ssy D}
\,\left(\,{\mathcal K}_{m}(\tau;x,y)
-G_{m}(\tau;x,y)\,\right)^2\,dydx\right)\,d\tau\,\right)^{\frac{1}{2}}\\
\leq&\,\left(\,\sum_{\ell=1}^{m}\int_{\ssy\Delta_{\ell}}
\left\|{\mathcal S}^{{m}-\ell+1}_{\ssy\dtau}
-{\mathcal S}(\tau_{m}-\tau)\right\|_{\ssy\rm HS}^2\;d\tau\,\right)^{\half},
\quad m=1,\dots,M.\\
\end{split}
\end{equation*}
Next, we introduce the following splitting
\begin{equation}\label{peykos2}
\max_{1\leq{m}\leq{\ssy M}}{\sf E}_{\ssy\sf TDR}^m\leq
\max_{1\leq{m}\leq{\ssy M}}{\mathcal A}_{m}
+\max_{1\leq{m}\leq{\ssy M}}{\mathcal B}_{m},
\end{equation}
with
\begin{equation*}
\begin{split}
{\mathcal A}_{m}:=&\,
\left(\,\sum_{\ell=1}^{m}\int_{\ssy\Delta_{\ell}}
\left\|{\mathcal S}^{{m}-\ell+1}_{\ssy\dtau}
-{\mathcal S}(\tau_{m}-\tau_{\ell-1})\right\|_{\ssy\rm HS}^2\;d\tau\,\right)^{\ssy 1/2},\\
{\mathcal B}_{m}:=&\,
\left(\,\sum_{\ell=1}^{m}\int_{\ssy\Delta_{\ell}}
\left\|{\mathcal S}(\tau_{m}-\tau_{\ell-1})
-{\mathcal S}(\tau_{m}-\tau)\right\|_{\ssy\rm HS}^2\;d\tau\,\right)^{\ssy 1/2}.\\
\end{split}
\end{equation*}
\par
Let $\delta\in\left[0,\frac{1}{2}\right)$. Then, using \eqref{September2015_0}
and \eqref{SR_BOUND}, we have
\begin{equation*}
\begin{split}
{\mathcal A}_{m}=&\,\left[\,\sum_{\kappa=1}^{\ssy\infty}\,\left(\,
\dtau\,\sum_{\ell=1}^{m}
\left\|{\mathcal S}^{m-\ell+1}_{\ssy\dtau}(\varepsilon_{\kappa})
-{\mathcal S}(\tau_{m-\ell+1})\varepsilon_{\kappa}\right\|_{\ssy 0,D}^2\,\right)
\,\right]^{\ssy 1/2}\\
=&\,\left[\,\sum_{\kappa=1}^{\ssy\infty}\,\left(\,
\dtau\,\sum_{\ell=1}^{m}
\left\|{\mathcal S}^{\ell}_{\ssy\dtau}(\varepsilon_{\kappa})
-{\mathcal S}(\tau_{\ell})\varepsilon_{\kappa}\right\|_{\ssy 0,D}^2
\,\right)\,\right]^{\ssy 1/2}\\
\leq&\,C\,\dtau^{\frac{\delta}{2}}\,\left(
\,\sum_{\kappa=1}^{\ssy\infty}\|\varepsilon_{\kappa}\|^2_{\ssy
{\bfdot H}^{\delta-1}}\,\right)^{\ssy 1/2}\\
\leq&\,C\,\dtau^{\frac{\delta}{2}}\,\left(
\,\sum_{\kappa=1}^{\ssy\infty}
\tfrac{1}{\lambda_{\kappa}^{2(1-\delta)}}\,\right)^{\ssy 1/2}\\
\leq&\,C\,\dtau^{\frac{\delta}{2}}\,\left(
\,\sum_{\kappa=1}^{\ssy\infty}
\tfrac{1}{\lambda_{\kappa}^{1+4(\frac{1}{4}-\frac{\delta}{2})}}\,\right)^{\ssy 1/2}\\
\leq&\,C\,\left(\tfrac{1}{4}-\tfrac{\delta}{2}\right)^{-\half}\,\dtau^{\frac{\delta}{2}},
\quad m=1,\dots,M,\\
\end{split}
\end{equation*}
which, after setting $\epsilon=\tfrac{1}{4}-\tfrac{\delta}{2}\in(0,\tfrac{1}{4}]$, yields
\begin{equation}\label{peykos3}
\max_{1\leq{m}\leq{\ssy M}}{\mathcal A}_m
\leq\,C\,\epsilon^{-\frac{1}{2}}\,\dtau^{\frac{1}{4}-\epsilon}.
\end{equation}
%
%
%
\par
Now, observing that
${\mathcal S}(t)\varepsilon_{\kappa}=e^{-\lambda_{\kappa}^2t}\,\varepsilon_{\kappa}$
for $t\ge 0$ and $\kappa\in{\mathbb N}$, we obtain
%
%
%
\begin{equation*}
\begin{split}
{\mathcal B}_m=&\,\left(\,\sum_{\kappa=1}^{\ssy\infty}
\sum_{\ell=1}^m\int_{\ssy\Delta_{\ell}} \left(\int_{\ssy D}\left[
e^{-\lambda_{\kappa}^2(\tau_m-\tau_{\ell-1})}
-e^{-\lambda_{\kappa}^2(\tau_m-\tau)}\right]^2
\varepsilon_{\kappa}^2(x)\,dx\right)
\,d\tau\,\right)^{\ssy 1/2}\\
=&\,\left(\,\sum_{\kappa=1}^{\ssy\infty}
\sum_{\ell=1}^m\int_{\ssy\Delta_{\ell}}
e^{-2\lambda_{\kappa}^2(\tau_m-\tau)} \left(1-
e^{-\lambda_{\kappa}^2(\tau-\tau_{\ell-1})}\right)^2
\,d\tau\,\right)^{\ssy 1/2}\\
\leq&\,\left(\,\sum_{\kappa=1}^{\ssy\infty}
\big(1-e^{-\lambda_{\kappa}^2\,\dtau}\big)^2
\,\int_0^{\tau_m} e^{2\lambda_{\kappa}^2(\tau-\tau_m)}
\,d\tau\,\right)^{\ssy 1/2}\\
\leq&\,\left(\,\sum_{\kappa=1}^{\ssy\infty}
\tfrac{\big(1-e^{-2\lambda_{\kappa}^2\,\dtau}\big)^2}
{\lambda_{\kappa}^2}\,\right)^{\ssy 1/2},
\quad m=1,\dots,M,\\
\end{split}
\end{equation*}
from which, applying \eqref{arodamos5}, we obtain
\begin{equation}\label{peykos4}
\max_{1\leq{m}\leq{\ssy M}}{\mathcal B}_m\leq\,\dtau^{\frac{1}{4}}.
\end{equation}
\par
Finally, the estimate \eqref{Dinner0} follows easily combining \eqref{peykos2},
\eqref{peykos3} and \eqref{peykos4}.
\end{proof}
%
%
\subsection{Estimating the space discretization error}
%
%
%
%
\begin{lemma}\label{prasinolhmma}
Let $f\in L^2(D)$ and $g_h$, $\psi_h$,  $z_h\in\fem$, such that
\begin{equation}\label{AuxEq_1}
\psi_h+\tfrac{\dtau}{2}\,\Delta_h\psi_h=P_hf
\quad \text{\it and}\quad
z_h=g_h-\tfrac{\dtau}{2}\,\Delta_hg_h.
\end{equation}
Then there exist functions $G_h$, ${\widetilde G}_h\in C({\overline {D\times D}})$
such that
\begin{gather}
\psi_h(x)=\int_{\ssy D} G_h(x,y)\,f(y)\,dy\quad\forall\,x\in{\overline D},
\label{prasinogreen1}\\
z_h(x)=\int_{\ssy D}{\widetilde G}_h(x,y)\,g_h(y)\,dy\quad\forall\,x\in{\overline D}.
\label{prasinogreen2}
\end{gather}
\end{lemma}
%
%
%
%
\begin{proof}
Let $\nu_h:=\text{\rm dim}(\fem)$ and
$\gamma:\fem\times\fem\rightarrow\rset$
be an inner product on $\fem$ defined by
$\gamma(\chi, \varphi) :=(\partial\chi,\partial\varphi)_{\ssy 0,D}$
for $\chi,\varphi\in\fem$.
Then, we can construct a basis $(\varphi_j)_{j=1}^{\ssy \nu_h}$ of $\fem$
which is $L^2(D)-$orthonormal, i.e.,
$(\varphi_i,\varphi_j)_{\ssy 0,D}=\delta_{ij}$ for
$i,j=1,\dots,\nu_h$,
and $\gamma-$orthogonal, i.e. there exist positive
$(\varepsilon_{h,j})_{j=1}^{\ssy \nu_h}$ such that
$\gamma(\varphi_i,\varphi_j)=\varepsilon_{h,i}\,\delta_{ij}$
for $i,j=1,\dots,\nu_h$
(see Sect. 8.7 in \cite{Golub}). Thus, there exist real numbers $(\mu_j)_{j=1}^{\nu_h}$
and $({\widetilde\mu}_j)_{j=1}^{\nu_h}$ such that $\psi_h=\sum_{j=1}^{\nu_h}\mu_j\,\varphi_j$
and $z_h=\sum_{j=1}^{\nu_h}{\widetilde\mu}_j\,\varphi_j$. Then, \eqref{AuxEq_1} yields
%
%
$\mu_{\ell}=\tfrac{1}{1+\frac{\dtau}{2}\,\varepsilon_{h,\ell}}\,\,(f,\varphi_{\ell})_{\ssy 0,D}$
and
${\widetilde\mu}_{\ell}=\left(1-\tfrac{\dtau}{2}\,\varepsilon_{h,\ell}\right)\,(g_h,\varphi_{\ell})_{\ssy 0,D}$
for $\ell=1,\dots,\nu_h$. Thus, we conclude \eqref{prasinogreen1} and
\eqref{prasinogreen2} with
$G_h(x,y)=\sum_{j=1}^{\nu_h}
\tfrac{1}{1+\frac{\dtau}{2}\,\varepsilon_{h,j}}\,\varphi_j(x)\,\varphi_j(y)$
and
${\widetilde G}_h(x,y)=\sum_{j=1}^{\nu_h}
(1-\tfrac{\dtau}{2}\,\varepsilon_{h,j})\,\varphi_j(x)\,\varphi_j(y)$
for $x$, $y\in{\overline D}$.
\end{proof}
%
%
%
%
\begin{proposition}\label{Tigrakis}
Let $(U_h^m)_{m=0}^{\ssy M}$ be the fully dicsrete approximations
defined by  \eqref{FullDE1}--\eqref{FullDE2} and $(U^m)_{m=0}^{\ssy M}$
be the time discrete approximations defined by \eqref{CN_S1}--\eqref{CN_S2}.
Then, there exists a constant $C_{\ssy\sf SDR}>0$,
independent of $\dt$, $\dx$, $\dtau$ and $h$, such that
\begin{equation}\label{Lasso1}
\max_{1\leq{m}\leq {\ssy M}}{\sf E}^m_{\ssy\sf SDR}
\leq\,C_{\ssy\sf SDR}\,\left(\,\epsilon_1^{-\half}\,\,\,h^{\frac{1}{2}-\epsilon_1}
+\epsilon_2^{-\frac{1}{2}}\,\dtau^{\frac{1}{4}-\epsilon_2}\right)
\quad\forall\,\epsilon_1\in\left(0,\tfrac{1}{2}\right],
\quad\forall\,\epsilon_2\in\left(0,\tfrac{1}{4}\right].
\end{equation}
\end{proposition}
%
%
%
%
\begin{proof}
In the sequel, we will use the symbol $C$ to denote a generic constant that is independent
of $\dt$, $\dx$, $\dtau$ and $h$, and may changes value from one line to the other.
\par
Let ${\sf I}:L^2(D)\to L^2(D)$ be the identity operator,
${\sf Y}_h:\fem\rightarrow\fem$ be defined
by ${\sf Y}_h:={\sf I}-\tfrac{\dtau}{2}\,\Delta_h$
and ${\sf\Lambda}_h:L^2(D)\to\fem$ be the inverse discrete elliptic
operator given by ${\sf\Lambda}_h:=({\sf I}+\tfrac{\dtau}{2}\,\Delta_h)^{-1}\circ P_h$.
Also, for $m=1,\dots,M$, we define a discrete operator
${\mathcal Q}_{h}^m:L^2(D)\rightarrow\fem$ by
${\mathcal Q}_{h}^{m}:=({\sf\Lambda}_h\circ{\sf Y}_h)^{m-1}\circ{\sf\Lambda}_h$,
which has a Green's function ${\mathcal G}_{{\mathcal Q}_h^m}$
(see Lemma~\ref{prasinolhmma}).
Also, for given $v_0\in{\bfdot H}^1(D)$, let
$({\mathcal S}_h^m(v_0))_{m=0}^{\ssy M}$
be fully discrete discrete approximations of the solution to the deterministic
problem \eqref{Det_Parab}, defined by \eqref{CNFD1}--\eqref{CNFD2}.
Then, using a simple induction argument, we conclude that
${\mathcal S}_h^m(v_0)={\mathcal Q}_h^{m}(v_0)$
for $m=1,\dots,M$.
\par
Applied an induction argument on \eqref{FullDE2}, we obtain
\begin{equation}\label{Anaparastash2}
\begin{split}
U_h^m(x)=&\,\sum_{\ell=1}^{\ssy m} \int_{\ssy\Delta_{\ell}}
{\mathcal Q}_h^{m-\ell+1}({\mathcal W}(\tau,x))\,d\tau\\
=&\int_0^{\ssy T}\int_{\ssy D}
\,{\mathcal K}_h^m(\tau;x,y)\,{\mathcal W}(\tau,y)
\,dyd\tau\quad\forall\,x\in{\overline D},
\end{split}
\end{equation}
where
${\mathcal K}_h^m(\tau;x,y)
:=\sum_{\ell=1}^m{\mathcal X}_{\ssy\Delta_{\ell}}(\tau)
\,{\mathcal G}_{{\mathcal Q}_h^{m-\ell+1}}(x,y)$ for  all
$\tau\in[0,T]$ and $x,y\in{\overline D}$.
%
%
%
Using \eqref{Anaparastash2}, \eqref{Dinner1}, \eqref{WNEQ2}, \eqref{Ito_Isom},
\eqref{HSxar} and \eqref{peykos1}, we get
\begin{equation*}
\begin{split}
{\sf E}^m_{\ssy\sf SDR}
\leq&\,\left(\,\int_0^{\tau_m}\left(\int_{\ssy D}\int_{\ssy D}
\,\left({\mathcal K}^{m}(\tau;x,y)
-{\mathcal K}_h^m(\tau;x,y)\right)^2
\,dydx\right)\,d\tau\,\right)^{\ssy 1/2}\\
\leq&\,\left(\,\dtau\sum_{\ell=1}^{m}
\left\|{\mathcal S}^{\ell}_{\ssy\dtau}
-{\mathcal S}_h^{\ell}\right\|_{\ssy\rm HS}^2\,\right)^{\ssy 1/2}.\\
\end{split}
\end{equation*}
%
%
Let $\delta\in\left[0,\tfrac{1}{2}\right)$ and $\theta\in\left[0,\frac{1}{4}\right)$.
Using \eqref{X_tiger_river1}, we obtain
\begin{equation*}
\begin{split}
%
%
\left(\,{\sf E}^m_{\ssy\sf  SDR}\,\right)^2
\leq&\,\sum_{k=1}^{\infty}
\,\left(\,\dtau\,\sum_{\ell=1}^{m}
\left\|{\mathcal S}^{\ell}_{\ssy\dtau}(\varepsilon_k)-{\mathcal S}_h^{\ell}(\varepsilon_k)
\right\|_{\ssy 0,D}^2\right)\\
%
\leq&\,C\,\left(\,h^{4\theta}
\,\sum_{k=1}^{\infty}
\|\varepsilon_k\|^2_{\ssy {\bfdot H}^{2\theta-1}}
+\dtau^{\delta}\,\sum_{k=1}^{\infty}
\|\varepsilon_k\|^2_{\ssy {\bfdot H}^{\delta-1}}\,\right)\\
\leq&\,C\,\left(\,h^{4\theta}
\,\sum_{k=1}^{\infty} \tfrac{1}{\lambda_k^{2\,(1-2\theta)}}
+\dtau^{\delta}\,
\sum_{k=1}^{\infty} \tfrac{1}{\lambda_k^{2\,(1-\delta)}}\,\right)\\
\leq&\,C\,\left(\,h^{4\theta}
\,\sum_{k=1}^{\infty} \tfrac{1}{\lambda_k^{1+2\,(\frac{1}{2}-2\theta)}}
+\dtau^{\delta}\,
\sum_{k=1}^{\infty} \tfrac{1}{\lambda_k^{1+4\,(\frac{1}{4}-\frac{\delta}{2})}}\,\right),
\end{split}
\end{equation*}
which, after using \eqref{SR_BOUND}, yields
\begin{equation}\label{Dec2015_10th}
{\sf E}^m_{\ssy\sf SDR}
\leq\,C\,\left[\,h^{2\theta}\,\left(\tfrac{1}{2}-2\theta\right)^{-\half}
+\dtau^{\frac{\delta}{2}}\,\left(\tfrac{1}{4}-\tfrac{\delta}{2}\right)^{-\half}\,\right].
\end{equation}
and
setting $\epsilon_1=\frac{1}{2}-2\theta\in\left(0,\frac{1}{2}\right]$ and
$\epsilon_2=\frac{1}{4}-\tfrac{\delta}{2}\in\left(0,\tfrac{1}{4}\right]$
we, easily, arrive at \eqref{Lasso1}.
\end{proof}
%
\subsection{Estimating the discretization error}
%
\begin{theorem}\label{FFQEWR}
Let ${\widehat u}$ be the solution to the problem \eqref{AC2},
and $(U_h^m)_{m=0}^{\ssy M}$ be the Crank-Nicolson
fully-discrete approximations of ${\widehat u}$ defined by
\eqref{FullDE1}-\eqref{FullDE2}. Then, there exists a constant
$C>0$, independent of  $\dx$, $\dt$, $h$ and $\dtau$,
such that
\begin{equation}\label{Final_Estimate}
\max_{0\leq{m}\leq{\ssy M}}\left({\mathbb E}\left[
 \|U_h^m-{\widehat u}(\tau_m,\cdot)\|_{\ssy 0,D}^2\right]
\right)^{\half} \leq\,C\,\left(\,
\epsilon_1^{-\frac{1}{2}}\,h^{\frac{1}{2}-\epsilon_1}
+\epsilon_2^{-\frac{1}{2}}\,\dtau^{\frac{1}{4}-\epsilon_2}\,\right)
\end{equation}
for all $\epsilon_1\in\left(0,\tfrac{1}{2}\right]$ and
$\epsilon_2\in\left(0,\tfrac{1}{4}\right]$.
\end{theorem}
%
%
\begin{proof}
The error estimate \eqref{Final_Estimate} is a simple consequence
of \eqref{split_split}, \eqref{Dinner0} and \eqref{Lasso1}.
\end{proof}
%
%
\begin{remark}
For an optimal, logarithmic-type choice of the parameter $\epsilon$ in \eqref{ModelError} and of
the parameters $\epsilon_1$ and $\epsilon_2$ in \eqref{Final_Estimate}, we refer the reader to
the discussion in Remark~3 of \cite{KZ2013a}.
\end{remark}
%
%
%
%
%
%
%
%
%
%
\def\cprime{$'$}
\def\cprime{$'$}
\end{document}